\documentclass{tac}
\usepackage[T1]{fontenc}
\usepackage[utf8]{inputenc}
\usepackage{amsmath}
\usepackage{cases}
\usepackage{amsfonts}
\usepackage{mathtools}
\usepackage{mathdots}
\usepackage[square]{natbib}
\usepackage{tikz}
\usepackage{pgf}
\usepackage{graphicx}
\usepackage{subcaption}
\usepackage{hyperref}
\usepackage{multicol}
\usepackage{float}
\usepackage{url}
\usepackage{ifthen}
\usetikzlibrary{arrows,automata}
\usetikzlibrary{positioning,calc,decorations.pathmorphing,decorations.pathreplacing,decorations.markings,patterns}
\newtheorem{lemma}{Lemma}

\newtheorem{thm}{Theorem}

\newcommand{\cC}[0]{%
\mathcal{C}}

\DeclareMathOperator{\dom}{dom}
\DeclareMathOperator{\cod}{cod}
\DeclareMathOperator{\domh}{domh}
\DeclareMathOperator{\domv}{domv}
\DeclareMathOperator{\codh}{codh}
\DeclareMathOperator{\codv}{codv}
\newcommand{\op}{\text{op}}
\newcommand{\Cat}[0]{%
\mathbf{Cat}}

\usetikzlibrary{calc,patterns,snakes,arrows,decorations.markings,shapes.misc}

\def\defaultfacecolor{green}

\tikzset{cross/.style={cross out, draw, 
         minimum size=2*(#1-\pgflinewidth), 
         inner sep=0pt, outer sep=0pt}}
\tikzstyle{bnode}=[draw,black,circle,fill=black,inner sep=1pt]
\tikzstyle{bplace}=[draw,blue,thick,cross,inner sep=2pt]
\tikzstyle{yspot}=[draw,\defaultfacecolor,circle,fill=\defaultfacecolor,inner sep=1pt]
\tikzstyle{eface}=[draw,\defaultfacecolor]
\tikzstyle{rnode}=[draw,red,circle,fill=red,inner sep=1pt]
\tikzstyle{dedge}=[postaction={nomorepostaction,decorate,
                    decoration={markings,mark=at position 0.5 with {\arrow{>}}}
}]
\tikzstyle{levelbar}=[dashed]

\makeatletter
\tikzset{nomorepostaction/.code={\let\tikz@postactions\pgfutil@empty}}
\makeatother

\newcommand\storeface[2]{\expandafter\xdef\csname faceid#1\endcsname{#2}}
\newcommand\getface[1]{\csname faceid#1\endcsname}
\newcounter{nextfaceid}
\newcounter{maxfaceseen}

\newcommand\storefinalface[2]{\expandafter\xdef\csname finalfaceid#1\endcsname{#2}}
\newcommand\getfinalface[1]{\csname finalfaceid#1\endcsname}

\newcommand\storefacecolor[2]{\expandafter\xdef\csname facecolor#1\endcsname{#2}}
\newcommand\getfacecolor[1]{\csname facecolor#1\endcsname}

\newcommand\storeminslice[2]{\expandafter\xdef\csname minfaceslice#1\endcsname{#2}}
\newcommand\getminslice[1]{\csname minfaceslice#1\endcsname}

\newcommand\storeedgeid[2]{\expandafter\xdef\csname edgeid#1\endcsname{#2}}
\newcommand\getedgeid[1]{\csname edgeid#1\endcsname}
\newcounter{nextedgeid}

\newcommand\storeedgepath[2]{\expandafter\xdef\csname edgepath#1\endcsname{#2}}
\newcommand\getedgepath[1]{\csname edgepath#1\endcsname}
\newcommand\expandedgepath[2]{\storeedgepath{\getedgeid{#1}}{\getedgepath{\getedgeid{#1}} #2}}

\newcommand\storeedgecolor[2]{\expandafter\xdef\csname edgecolor#1\endcsname{#2}}
\newcommand\getedgecolor[1]{\csname edgecolor#1\endcsname}

\newcommand\mergefaces[2]{
   \pgfmathtruncatemacro\valfa{\getfinalface{#1}}
   \pgfmathtruncatemacro\valfb{\getfinalface{#2}}
   \pgfmathtruncatemacro\newfaceid{min(\valfa, \valfb)}
   \pgfmathtruncatemacro\latestfaceid{\themaxfaceseen - 1}
   \foreach \x in {0,...,\latestfaceid}{
       \ifthenelse{\getfinalface{\x}=\valfa \OR \getfinalface{\x}=\valfb}{
          \storefinalface{\x}{\newfaceid}
       }{}
   }     
}

\def\drawfaces{0}

\def\writefaceids{0}

\def\drawplaces{0}

\newcommand\startdiagram[1]{
   \pgfmathtruncatemacro\maxstrandidx{#1 - 1}
   \pgfmathsetmacro\centeringoffset{-0.5 * #1}
   \pgfmathtruncatemacro\diagramlevel{0}

   \foreach \x in {0,...,\maxstrandidx}{
      \storeface{\x}{\x}
      \storefinalface{\x}{\x}
      \storefacecolor{\x}{\defaultfacecolor}
      \storeminslice{\x}{0}
   }
   \pgfmathtruncatemacro\tmpnextfaceid{\maxstrandidx + 1}
   \setcounter{nextfaceid}{\tmpnextfaceid}
   \setcounter{maxfaceseen}{\tmpnextfaceid}

   \setcounter{nextedgeid}{0}
}

\newcommand\drawinitialstrands[1]{
   \pgfmathtruncatemacro\maxstrandidx{#1 - 1}
   \pgfmathsetmacro\centeringoffset{-0.5 * #1}
   \pgfmathtruncatemacro\diagramlevel{0}

   \foreach \x in {0,...,\maxstrandidx}{
      \storeface{\x}{\x}
   }

  \begin{scope}[xshift=\centeringoffset cm]
   \ifthenelse{\maxstrandidx = 0}{}{
   \foreach \x in {1,...,\maxstrandidx}{

     \storeedgeid{\x}{\thenextedgeid}
     \storeedgepath{\thenextedgeid}{($(\x + \centeringoffset,.5)$) -- ($(\x+\centeringoffset,0)$)};
     \storeedgecolor{\thenextedgeid}{black}
     \stepcounter{nextedgeid}
   }
   }
   \end{scope}

    \foreach \x in {0,...,\maxstrandidx} {
        \node at ($(\x+\centeringoffset + .5, 0-\diagramlevel)$) (spot_\diagramlevel_\x) {};
    }
    \foreach \x in {1,...,\maxstrandidx} {
      \node at ($(\x+\centeringoffset, 0-\diagramlevel)$) (place_\diagramlevel_\x) {};
    }

   \pgfmathtruncatemacro\tmpnextfaceid{\maxstrandidx + 1}
   \setcounter{nextfaceid}{\tmpnextfaceid}

   \pgfmathtruncatemacro\latestfaceid{\themaxfaceseen - 1}

}

\newcommand\finishdiagram[0]{
   \ifthenelse{\maxstrandidx = 0}{}{
     \foreach \x in {1,...,\maxstrandidx}{
       \expandedgepath{\x}{ -- +(0,-.5)};
     
   }
   }

   \pgfmathtruncatemacro\maxedgeid{\thenextedgeid - 1}
   \ifthenelse{\thenextedgeid = 0}{}{
   \foreach \eid in {0,...,\maxedgeid}{
     \draw[rounded corners,\getedgecolor{\eid}] \getedgepath{\eid} ;
   }
   }
}

\newcommand\scanslice[3]{
    \def\wb{#1}
    \def\inputs{#2}
    \def\outputs{#3}
    \pgfmathtruncatemacro\nextdiaglevel{\diagramlevel+1}
    \pgfmathtruncatemacro\horizoffset{\outputs - \inputs}
    \pgfmathtruncatemacro\bottomrightcorner{\wb + \outputs + 1}
    \pgfmathtruncatemacro\innertoprightcorner{\wb + \inputs}

    \ifthenelse{\outputs=0 \AND \inputs>0}{
      \mergefaces{\getface{\wb}}{\getface{\innertoprightcorner}}
    }{}

   \ifthenelse{\horizoffset > 0}{
       \foreach \x in {\maxstrandidx,...,\innertoprightcorner}{
          \pgfmathtruncatemacro\offsetidx{\x + \horizoffset}
          \storeface{\offsetidx}{\getface{\x}}
       }
   }{
       \foreach \x in {\innertoprightcorner,...,\maxstrandidx}{
          \pgfmathtruncatemacro\offsetidx{\x + \horizoffset}
          \storeface{\offsetidx}{\getface{\x}}
       }
   }

   \pgfmathtruncatemacro\firstnewface{\wb + 1}
   \pgfmathtruncatemacro\finalnewface{\wb + \outputs - 1}
   \ifthenelse{\outputs > 1}{
     \foreach \x in {\firstnewface,...,\finalnewface}{
       \storeface{\x}{\thenextfaceid}

       \storeminslice{\thenextfaceid}{\nextdiaglevel}

       \ifthenelse{\themaxfaceseen > \thenextfaceid}{}{
         \storefinalface{\thenextfaceid}{\thenextfaceid}
         \storefacecolor{\thenextfaceid}{eface}
         \stepcounter{maxfaceseen}
       }
       \stepcounter{nextfaceid}

     }
   }{}

   \pgfmathtruncatemacro\latestfaceid{\themaxfaceseen - 1}

    \pgfmathtruncatemacro\maxstrandidx{\maxstrandidx + \horizoffset}
    \pgfmathtruncatemacro\diagramlevel{\diagramlevel + 1}

}

\newcommand\diagslicenovertex[3]{
    \def\wb{#1}
    \def\inputs{#2}
    \def\outputs{#3}
    \pgfmathtruncatemacro\nextdiaglevel{\diagramlevel+1}
    \pgfmathtruncatemacro\horizoffset{\outputs - \inputs}
    \pgfmathsetmacro\nextoffset{\centeringoffset - 0.5*\horizoffset}
    \pgfmathtruncatemacro\toprightcorner{\wb + \inputs + 1}
    \pgfmathtruncatemacro\bottomrightcorner{\wb + \outputs + 1}
    \pgfmathsetmacro\vertexpos{0.5* (\centeringoffset + \wb + 0.5 + 0.5*\inputs) + 0.5*(\nextoffset + \wb + 0.5 + 0.5*\outputs)}

    \node (curvertex) at ($(\vertexpos, -0.5*\diagramlevel - 0.5*\nextdiaglevel)$) {};
    \node at (curvertex) (v\diagramlevel) {};

    \ifthenelse{\wb=0}{}{
        \foreach \x in {1,...,\wb} {

          \expandedgepath{\x}{-- ($(\x + \nextoffset,0- \nextdiaglevel)$)};
        }
    }
    
    \ifthenelse{\inputs=0}{}{
      \foreach \x in {1,...,\inputs} {
         \pgfmathtruncatemacro\edgepos{\wb+\x}

         \expandedgepath{\edgepos}{-- ($(\vertexpos, -0.5*\diagramlevel - 0.5*\nextdiaglevel)$)};
       }
    }

    \ifthenelse{\toprightcorner > \maxstrandidx}{}{
       \foreach \x in {\toprightcorner,...,\maxstrandidx} {
         \expandedgepath{\x}{-- ($(\x + \horizoffset + \nextoffset, 0 - \nextdiaglevel)$)};
         }

         \ifthenelse{\inputs > \outputs}{
           \foreach \x in {\toprightcorner,...,\maxstrandidx}{
              \pgfmathtruncatemacro\edgepos{\x + \horizoffset}
              \storeedgeid{\edgepos}{\getedgeid{\x}}
           }
         }{}
         \ifthenelse{\outputs > \inputs}{
           \foreach \x in {\maxstrandidx,...,\toprightcorner}{
              \pgfmathtruncatemacro\edgepos{\x + \horizoffset}
              \storeedgeid{\edgepos}{\getedgeid{\x}}
           }           
         }{}
    }

    \ifthenelse{\outputs=0}{}{
      \foreach \x in {1,...,\outputs} {
        \pgfmathtruncatemacro\edgepos{\wb + \x}
        \storeedgeid{\edgepos}{\thenextedgeid}
        \storeedgepath{\thenextedgeid}{($(\vertexpos, -0.5*\diagramlevel - 0.5*\nextdiaglevel)$) -- ($(\edgepos + \nextoffset,0 - \nextdiaglevel)$)};
        \storeedgecolor{\thenextedgeid}{black}
        \stepcounter{nextedgeid}
       }
    }

    \ifthenelse{\drawfaces=0}{}{
         \ifthenelse{\inputs=0 \AND \outputs=0}{
            \pgfmathtruncatemacro\innertopleftcorner{\wb - 1}
            \pgfmathtruncatemacro\innertoprightcorner{\toprightcorner}

            \draw[eface] ($(\wb + \centeringoffset + .5, 0 - \diagramlevel)$) edge[bend left=40] ($(\wb + \nextoffset + .5, 0 - \nextdiaglevel)$);
            \draw[eface] ($(\wb + \centeringoffset + .5, 0 - \diagramlevel)$) edge[bend right=40] ($(\wb + \nextoffset + .5, 0 - \nextdiaglevel)$);
        }{
            \pgfmathtruncatemacro\innertopleftcorner{\wb}
            \pgfmathtruncatemacro\innertoprightcorner{\toprightcorner - 1}
        }

        \foreach \x in {0,...,\innertopleftcorner} {
                 \draw[eface,\getfacecolor{\getfinalface{\getface{\x}}}] ($(\x + \centeringoffset + .5, 0 - \diagramlevel)$) -- ($(\x + \nextoffset + .5, 0 - \nextdiaglevel)$);
        }
         
         \foreach \x in {\innertoprightcorner,...,\maxstrandidx} {
            \draw[eface,\getfacecolor{\getfinalface{\getface{\x}}}] ($(\x + \centeringoffset + .5, 0 - \diagramlevel)$) -- ($(\x + \horizoffset + \nextoffset + .5, 0 - \nextdiaglevel)$);
         }

    }

    \scanslice{#1}{#2}{#3}
    \pgfmathsetmacro\centeringoffset{\nextoffset}

    \foreach \x in {0,...,\maxstrandidx} {
        \node at ($(\x+\centeringoffset + .5, 0-\diagramlevel)$) (spot_\diagramlevel_\x) {};
    }
    \foreach \x in {1,...,\maxstrandidx} {
      \node at ($(\x+\centeringoffset, 0-\diagramlevel)$) (place_\diagramlevel_\x) {};
    }
    \ifthenelse{\drawfaces=0}{}{
       \placesandspots
    }

}

\newcommand\diagslice[3]{
  \diagslicenovertex{#1}{#2}{#3}
  \node[bnode] at (curvertex) {};
}

\newcommand\placesandspots[0]{
    \foreach \x in {0,...,\maxstrandidx} {
        \node[yspot,\getfacecolor{\getfinalface{\getface{\x}}}] at (spot_\diagramlevel_\x) {};
        
       \ifthenelse{\writefaceids=1 \AND \diagramlevel=\getminslice{\getface{\x}} \AND \getfinalface{\getface{\x}}=\getface{\x}}{
            \node[node distance=.25cm,above of=spot_\diagramlevel_\x] {\small \getface{\x}};
        }{}
    }

    \ifthenelse{\maxstrandidx=0 \OR \drawplaces=0}{}{
       \foreach \x in {1,...,\maxstrandidx} {
          \node[bplace] at (place_\diagramlevel_\x) {};
       }
    }
}

\title{The word problem for double categories}
\author{Antonin Delpeuch}
\address{Department of Computer Science \\ University of Oxford}
\eaddress{antonin.delpeuch@cs.ox.ac.uk}
\date{\today}

\begin{document}

\maketitle

\begin{abstract}
  We solve the word problem for free double categories without
  equations between generators by translating it to the word problem
  for 2-categories. This yields a quadratic algorithm deciding the
  equality of diagrams in a free double category. The translation is
  of interest in its own right since and can for instance be used to
  reason about double categories with the language of 2-categories,
  sidestepping the pinwheel problem. It also shows that although
  double categories are formally more general than 2-categories, they are
  not actually more expressive, explaining the rarity of applications of this
  notion.
\end{abstract}

\section*{Introduction}

The combinatorial structure of double categories has
attracted a lot of attention since the notion was introduced by
\cite{ehresmann1963categories}.  Informally, one can describe double
categories by the shape of their string diagrams.  Unlike 2-categories
where the edges are required to flow along a specified
direction (usually vertically), 2-cells in double categories can
connect to both horizontal and vertical wires. Therefore they not only
have a vertical domain and codomain, but also a horizontal domain and
codomain. These definitions are made precise in
Section~\ref{sec:double-categories}.

\begin{figure}[H]
  \centering
  \begin{subfigure}{0.4\textwidth}
    \centering
    \begin{tikzpicture}[scale=1.5]
      \definecolor{c1}{RGB}{212,166,204}
      \definecolor{c2}{RGB}{152,209,162}
      \definecolor{c3}{RGB}{248,168,156}
      \definecolor{c4}{RGB}{243,237,155}
      \node[inner sep=0pt] at (.5,.5) (x) {};
      \path[fill=c1] (x) .. controls (.3,.2) and (.3,.2) .. (.3,0) -- (.7,0) .. controls (.7,.2) and (.7,.2) .. (x);


      \path[fill=c2] (x) .. controls (.3,.2) and (.3,.2) .. (.3,0) -- (0,0) -- (0,1) -- (.5,1) -- (x);

      \path[fill=c4] (x) .. controls (.7,.2) and (.7,.2) .. (.7,0) -- (1,0) -- (1,1) -- (.5,1) -- (x);

      \draw[decoration={markings,mark=at position 0.7 with {\arrow{>}}},postaction={decorate}] (x) .. controls (.3,.2) and (.3,.2) .. (.3,0);

      \draw[decoration={markings,mark=at position 0.7 with {\arrow{>}}},postaction={decorate}] (x) .. controls (.7,.2) and (.7,.2) .. (.7,0);

      \draw[decoration={markings,mark=at position 0.7 with {\arrow{<}}},postaction={decorate}] (x) -- (.5,1);
      \draw[gray] (0,0) -- (0,1) -- (1,1) -- (1,0) -- (0,0);

      \node[circle,fill=white,draw,inner sep=1.5pt,scale=.8] at (x) {$\alpha$};

    \end{tikzpicture}
    \caption{A morphism in a 2-category}
  \end{subfigure}
  \begin{subfigure}{0.45\textwidth}
    \centering
    \begin{tikzpicture}[scale=1.5]
      \definecolor{c1}{RGB}{212,166,204}
      \definecolor{c2}{RGB}{152,209,162}
      \definecolor{c3}{RGB}{248,168,156}
      \definecolor{c4}{RGB}{243,237,155}
      \definecolor{c5}{RGB}{190,190,192}
      \definecolor{c6}{RGB}{163,219,227}
      \definecolor{c7}{RGB}{221,164,119}
      \definecolor{c8}{RGB}{200,230,180}
      
      \node[inner sep=0pt] at (.5,.5) (x) {};
      \path[fill=c1] (x) .. controls (.3,.2) and (.3,.2) .. (.3,0) -- (.7,0) .. controls (.7,.2) and (.7,.2) .. (x);

      \path[fill=c3] (x) .. controls (.3,.8) and (.3,.8) .. (.3,1) -- (.7,1) .. controls (.7,.8) and (.7,.8) .. (x);

      \path[fill=c2] (x) .. controls (.3,.2) and (.3,.2) .. (.3,0) -- (0,0) -- (0,.5) -- (x);
      \path[fill=c4] (x) .. controls (.3,.8) and (.3,.8) .. (.3,1) -- (0,1) -- (0,.5) -- (x);

      \path[fill=c7] (x) .. controls (.7,.2) and (.7,.2) .. (.7,0) -- (1,0) -- (1,.3) .. controls (.8,.3) and (.8,.3) .. (x);
      \path[fill=c8] (x) .. controls (.7,.8) and (.7,.8) .. (.7,1) -- (1,1) -- (1,.7) .. controls (.8,.7) and (.8,.7) .. (x);

      \path[fill=c5] (x) .. controls (.8,.3) and (.8,.3) .. (1,.3) -- (1,.5) -- (x);

      \path[fill=c6] (x)  -- (1,.5) -- (1,.7) .. controls (.8,.7) and (.8,.7) .. (x);

      \tikzset{fwd/.style={decoration={markings,mark=at position 0.7 with {\arrow{>}}},postaction={decorate}}}
            \tikzset{fwdmid/.style={decoration={markings,mark=at position 0.5 with {\arrow{>}}},postaction={decorate}}}
      \tikzset{bwd/.style={decoration={markings,mark=at position 0.7 with {\arrow{<}}},postaction={decorate}}}
      
      \draw[fwd] (x) .. controls (.3,.2) and (.3,.2) .. (.3,0);
      \draw[fwdmid] (0,.5) -- (x);
            \draw[fwd] (x) -- (1,.5);

      \draw[bwd] (x) .. controls (.3,.8) and (.3,.8) .. (.3,1);

      \draw[fwd] (x) .. controls (.7,.2) and (.7,.2) .. (.7,0);

      \draw[bwd] (x) .. controls (.7,.8) and (.7,.8) .. (.7,1);

      \draw[fwd] (x) .. controls (.8,.3) and (.8,.3) .. (1,.3);


      \draw[fwd] (x) .. controls (.8,.7) and (.8,.7) .. (1,.7);

      \draw[gray] (0,0) -- (0,1) -- (1,1) -- (1,0) -- (0,0);

      \node[circle,fill=white,draw,inner sep=1.5pt,scale=.8] at (x) {$\alpha$};
    \end{tikzpicture}
    \caption{A morphism in a double category}
  \end{subfigure}
\end{figure}

At a first glance, double categories could be considered a more
natural categorical axiomatization of planar systems, since they treat
the two dimensions of the plane in a dual, interchangeable way. In
comparison, the vertical and horizontal compositions in 2-categories
are intrisically different, forcing diagrams to flow in a specified
direction.  However, this uniform behaviour in two dimensions comes at a cost
known as the \emph{pinwheel problem}.  Concretely, this problem manifests
itself in the fact that not all planar arrangements of 2-cells can be
composed, even if all local compatibility conditions are satisfied.
For a diagram to be interpreted as a 2-cell it must be binary
composable and this can fail if the diagram contains a so-called
pinwheel, represented later in Figure~\ref{fig:pinwheel}.

A lot of work has already been dedicated to characterizing which
arrangements of 2-cells can be composed in a double category, using
order-theoretic representations of these
arragements~\citep{dawson1993characterizing,dawson1995forbiddensuborder}.
In this work, we focus instead on the word problem for 2-cells in
double categories. Given two binary composable diagrams, we want to
determine whether they represent the same 2-cell or not. \cite{dawson2004free}
have studied this problem in the case of free extensions of double categories,
showing for instance that the word problem can become undecidable with the addition
of a single free 2-cell.

We study the word problem for free double categories, meaning that
no equations are imposed on the generators. The only equations
relating expressions in this context are the axioms of double categories.
We introduce a correspondence between a free double category and a free
2-category, for which the word problem is
solved~\citep{delpeuch2018normalization}. We obtain as a result a
quadratic time algorithm to determine if two double category diagrams
are equivalent (Theorem~\ref{thm:word-problem}).

It might be worth explaining why this word problem is of interest.
Word problems have not attracted much attention in category theory so far,
perhaps because they are wrongly seen as computational problems of little relevance
to mathematics. In fact, studying the word problem for an algebraic structure
often surfaces profound algebraic properties of that structure, and is 
key to understanding its combinatorics. For instance, \citet{squier1987word}
established a correspondence between the existence of a convergent presentation
for a monoid and its homological type. Group theory in another example:
the notion of \emph{automatic group}~\citep{epstein1992word} was originally
motivated by computational properties, but turned out to have a simple
characterization in terms of Cayley graphs (ibid.). It has found fruitful applications
to the braid group~\citep{charney1992artin} and mapping class groups in general~\citep{mosher1995mapping},
among others.

Our solution to the word problem for double categories relies on a
reduction to the word problem for 2-categories. Here again, the translation used
is of its own interest, as it establishes a tight relation between the
combinatorics of 2-categories and that of double categories. In fact, we
will argue in Section~\ref{sec:pinwheel} that free 2-categories should
be preferred to free double categories, as they are simpler, equally expressive
and do not suffer from the pinwheel problem.

The idea of the correspondence is very simple. In order to simulate
the horizontal wires of a double category in a 2-category, we simply
``rotate the string diagrams by $\frac{\pi}{4}$''.  In
Section~\ref{sec:translation}, we make this correspondence precise and
show that it respects the notions of equivalences on both structures.
This lets us solve the word problem for free double categories in
Section~\ref{sec:word-problem}.

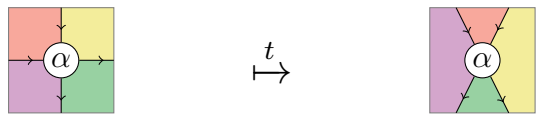
\begin{center}
\begin{tikzpicture}[scale=1.4]
  \definecolor{c1}{RGB}{212,166,204}
  \definecolor{c2}{RGB}{152,209,162}
  \definecolor{c3}{RGB}{248,168,156}
  \definecolor{c4}{RGB}{243,237,155}
  \tikzset{fwd/.style={decoration={markings,mark=at position 0.7 with {\arrow{>}}},postaction={decorate}}}
  \tikzset{bwd/.style={decoration={markings,mark=at position 0.5 with {\arrow{<}}},postaction={decorate}}}

  \path[fill=c1] (0,0) rectangle (.5,.5);
  \path[fill=c2] (.5,0) rectangle (1,.5);
  \path[fill=c3] (0,.5) rectangle (.5,1);
  \path[fill=c4] (.5,.5) rectangle (1,1);
  \draw[gray] (0,0) rectangle (1,1);
  \node[circle,draw,inner sep=1.5pt,fill=white] at (.5,.5) (alpha) {$\alpha$};
  \draw[fwd] (0,.5) -- (alpha);
  \draw[fwd] (alpha) -- (1,.5);
  \draw[bwd] (.5,0) -- (alpha);
  \draw[bwd] (alpha) -- (.5,1);

  \node[scale=1.3] at (2.5,.5) {$\xmapsto{t}$};

  \begin{scope}[xshift=4cm]
    \path[fill=c1] (0,0) -- (.25,0) -- (.5,.5) -- (.25,1) -- (0,1) -- (0,0);
    \path[fill=c3] (.25,1) -- (.5,.5) -- (.75,1);
    \path[fill=c4] (.75,1) -- (1,1) -- (1,0) -- (.75,0) -- (.5,.5);
    \path[fill=c2] (.75,0) -- (.5,.5) -- (.25,0);
    \draw[gray] (0,0) rectangle (1,1);
    \node[circle,draw,inner sep=1.5pt,fill=white] at (.5,.5) (alpha) {$\alpha$};
    \draw[fwd] (.25,1) -- (alpha);
    \draw[fwd] (alpha) -- (.75,0);
    \draw[bwd] (.25,0) -- (alpha);
    \draw[bwd] (alpha) -- (.75,1);
  \end{scope}
\end{tikzpicture}
\end{center}

This correspondence between free double categories and free
2-categories is motivated by the word problem but is of interest in
its own right: it shows that one does not gain much by considering a
free double category instead of the corresponding free
2-category. Reasoning in a 2-category avoids the pinwheel problem
entirely as the validity of a string diagram in this structure can be
checked locally. Section~\ref{sec:pinwheel} shows how the translation
could be extended to diagrams which include pinwheels, giving them
a meaning in the free 2-category.
This has also practical implications: one can use the
translation to reason about double categories in proof assistants such
as \href{https://homotopy.io/}{homotopy.io}~\citep{homotopyio-tool}
which use a globular notion of n-category.

\section*{Acknowledgements}

We thank Jamie Vicary, Jules Hedges, Robert Paré, Dorette Pronk and the anonymous reviewer for their feedback on this work.
The author is supported by an EPSRC studentship.

\section{Double categories} \label{sec:double-categories}

\begin{definition}
  Let $\cC$ be a category. An \textbf{internal category} in $\cC$
  consists of the following data:
  \begin{itemize}
   \item a pair of objects $M, O$, that we think of as the sets of morphisms and objects
   \item morphisms $d,c \in \cC(M,O)$, intuitively the domain and codomain functions
   \item a morphism $\iota \in \cC(O,M)$, taking an object to its identity map;
   \item a morphism $\mu \in \cC(P,M)$, where $P$ is the pullback
     (which is therefore required to exist) of $M \xrightarrow{d} O \xleftarrow{c} M$.
     This represents the multiplication of compatible pairs of morphisms.
  \end{itemize}
  These morphisms are required to satisfy equalities, which correspond to the axioms
  of a category (associativity and unitality of composition, as well as equations for the domains
  and codomains of identities and composites).
\end{definition}

The definition above is chosen such that an internal category in
$\mathbf{Set}$ is a small category. The purpose of this concept is
that its generality makes it possible to interpret it in other
categories. 

\begin{definition}
  A \textbf{double category} is an internal category in
  $\mathbf{Cat}$, the category of small categories.
\end{definition}

This definition is concise and this conciseness justifies the interest
in this structure, which was originally introduced
by~\cite{ehresmann1963categories}. However, it is of little help to
build intuition about the nature of such an object, so let us unfold
its content. A double category consists of an object category
$\mathcal{O}$ and a morphisms category $\mathcal{M}$, with functors
$D, C : \mathcal{M} \rightarrow \mathcal{O}$, $I : \mathcal{O}
\rightarrow \mathcal{M}$ and $M : \mathcal{P} \rightarrow \mathcal{M}$
where $\mathcal{P}$ is defined as above.  We will call
\begin{itemize}
\item objects of $\mathcal{O}$ as \textbf{objects} of the double category;
\item morphisms of $\mathcal{O}$ as \textbf{vertical morphisms} of the double category;
\item objects of $\mathcal{M}$ as \textbf{horizontal morphisms} of the double category;
\item morphisms of $\mathcal{M}$ as \textbf{2-cells} of the double category.
\end{itemize}

Initially, it can seem confusing that objects of $\mathcal{M}$ are
thought of as morphisms.  The reason for this is that by forgetting
morphisms, i.e. taking the image of our internal category via the
forgetful functor $\mathbf{Cat} \rightarrow \mathbf{Set}$, we obtain
an internal category in $\mathbf{Set}$, i.e. a small category. We will
call this the \textbf{horizontal category} of the double category. As
its morphisms are the objects of $\mathcal{M}$, this justifies their name.
These horizontal morphisms have as domains and codomains objects of $\mathcal{O}$.
These objects are also involved in another category, namely $\mathcal{O}$
itself, that we will call the \textbf{vertical category} of the double category.

Any 2-cell $\alpha$ has two horizontal morphisms as domain and
codomain, $\text{dom}_{\mathcal{M}}(\alpha)$ and
$\text{cod}_{\mathcal{M}}(\alpha)$ as a morphism of $\mathcal{M}$. We
will call these the \textbf{horizontal domain} and \textbf{codomain}
of $\alpha$.  Furthermore, it is associated by the internal category
structure to $D(\alpha)$ and $C(\alpha)$, which are vertical
morphisms. We will therefore call these the \textbf{vertical domain}
and \textbf{codomain} of $\alpha$. Finally, the functoriality of $D$ and $C$
ensures that for instance $D(\dom_{\mathcal{M}}(\alpha)) = \dom_{\mathcal{O}}(D(\alpha))$ and similarly
for $C$ and $\cod$. This suggests the representation of $\alpha$ as
a square:

\begin{figure}[H]
  \centering
  \begin{tikzpicture}[scale=2.5]

    \node at (0,1) (a) {$A$};
    \node at (1,1) (b) {$B$};
    \node at (0,0) (c) {$E$};
    \node at (1,0) (d) {$F$};

    \node at (.5,.5) (alpha) {$\alpha$};

    \draw[-latex] (a) edge node[above] {$\dom_{\mathcal{M}}(\alpha)$} (b);
    \draw[-latex] (c) edge node[below] {$\cod_{\mathcal{M}}(\alpha)$} (d);
    \draw[-latex] (a) edge node[left] {$D(\alpha)$} (c);
    \draw[-latex] (b) edge node[right] {$C(\alpha)$} (d);
  \end{tikzpicture}
\end{figure}
Although this diagram is similar to commutative diagrams used in category theory,
we stress that it is here used in a more general sense, as composing horizontal and
vertical morphisms does not make sense in general.

The 2-cells in a double category can be composed in two different ways.
First, as morphisms of $\mathcal{M}$, two 2-cells can be composed if they
have compatible horizontal domain and codomain. We call this the \textbf{vertical composition}.
Second, the functor $M$
defines a composition for 2-cells with compatible vertical domains and codomain, and we call this
the \textbf{horizontal composition}. These compositions can be represented with diagrams:
\begin{figure}[H]
  \centering
\begin{tikzpicture}[scale=1.8,every node/.style={scale=0.8}]

  \node at (0,1) (a) {$A$};
  \node at (1,1) (b) {$B$};
  \node at (0,0) (c) {$U$};
  \node at (1,0) (d) {$V$};

  \node at (.5,.5) (alpha) {$\alpha \circ \beta$};

  \draw[-latex] (a) edge node[above] {$\dom(\alpha \circ \beta)$} (b);
  \draw[-latex] (c) edge node[below] {$\cod(\alpha \circ \beta)$} (d);
  \draw[-latex] (a) edge node[left] {$D(\alpha \circ \beta)$} (c);
  \draw[-latex] (b) edge node[right] {$C(\alpha \circ \beta)$} (d);

  \node at (2,.5) {$=$};

  \begin{scope}[xshift=2.6cm,yshift=.5cm]
    \node at (0,1) (a) {$A$};
    \node at (1,1) (b) {$B$};
    \node at (0,0) (c) {$E$};
    \node at (1,0) (d) {$F$};
    \node at (0,-1) (u) {$U$};
    \node at (1,-1) (v) {$V$};

    \node at (.5,.5) (alpha) {$\beta$};
    \node at (.5,-.5) (beta) {$\alpha$};

    \draw[-latex] (a) edge node[above] {$\dom(\beta)$} (b);
    \draw[-latex] (u) edge node[below] {$\cod(\alpha)$} (v);
    \draw[-latex] (c) edge (d);
    \draw[-latex] (a) edge node[left] {$D(\beta)$} (c);
    \draw[-latex] (b) edge node[right] {$C(\beta)$} (d);
    \draw[-latex] (c) edge node[left] {$D(\alpha)$} (u);
    \draw[-latex] (d) edge node[right] {$C(\alpha)$} (v);

  \end{scope}
\end{tikzpicture}
\begin{tikzpicture}[scale=1.8,every node/.style={scale=0.8}]

  \node at (0,1) (a) {$A$};
  \node at (1,1) (b) {$B$};
  \node at (0,0) (c) {$U$};
  \node at (1,0) (d) {$V$};

  \node at (.5,.5) (alpha) {$M(\alpha, \beta)$};

  \draw[-latex] (a) edge node[above] {$\dom(M(\alpha, \beta))$} (b);
  \draw[-latex] (c) edge node[below] {$\cod(M(\alpha, \beta))$} (d);
  \draw[-latex] (a) edge node[left] {$D(M(\alpha, \beta))$} (c);
  \draw[-latex] (b) edge node[right] {$C(M(\alpha, \beta))$} (d);

  \node at (2.1,.5) {$=$};

  \begin{scope}[xshift=2.8cm]
    \node at (0,1) (a) {$A$};
    \node at (1,1) (b) {$B$};
    \node at (0,0) (c) {$E$};
    \node at (1,0) (d) {$F$};
    \node at (2,1) (u) {$U$};
    \node at (2,0) (v) {$V$};

    \node at (.5,.5) (alpha) {$\alpha$};
    \node at (1.5,.5) (beta) {$\beta$};

    \draw[-latex] (a) edge node[above] {$\dom(\alpha)$} (b);
    \draw[-latex] (c) edge node[below] {$\cod(\alpha)$} (d);
    \draw[-latex] (a) edge node[left] {$D(\alpha)$} (c);
    \draw[-latex] (b) edge (d);
    \draw[-latex] (u) edge node[right] {$C(\beta)$} (v);
    \draw[-latex] (b) edge node[above] {$\dom(\beta)$} (u);
    \draw[-latex] (d) edge node[below] {$\cod(\beta)$} (v);

  \end{scope}
\end{tikzpicture}
\end{figure}

The functoriality of $M$ ensures that these two compositions are compatible: $(\alpha \star \delta) \circ (\beta \star \gamma) = (\alpha \circ \beta) \star (\delta \circ \gamma)$ for all 2-cells such that both sides of the equation are defined. This means that the following diagram is unambiguous:
\[
\begin{tikzpicture}[scale=1.3]
  \node at (0,2) (a) {$A$};
  \node at (1,2) (b) {$B$};
  \node at (2,2) (c) {$C$};
  \node at (0,1) (d) {$D$};
  \node at (1,1) (e) {$E$};
  \node at (2,1) (f) {$F$};
  \node at (0,0) (g) {$G$};
  \node at (1,0) (h) {$H$};
  \node at (2,0) (i) {$I$};

  \draw[-latex] (a) -- (b);
  \draw[-latex] (b) -- (c);
  \draw[-latex] (d) -- (e);
  \draw[-latex] (e) -- (f);
  \draw[-latex] (g) -- (h);
  \draw[-latex] (h) -- (i);

  \draw[-latex] (a) -- (d);
  \draw[-latex] (d) -- (g);
  \draw[-latex] (b) -- (e);
  \draw[-latex] (e) -- (h);
  \draw[-latex] (c) -- (f);
  \draw[-latex] (f) -- (i);

  \node at (.5,1.5) {$\beta$};
  \node at (1.5,1.5) {$\gamma$};
  \node at (.5,.5) {$\alpha$};
  \node at (1.5,.5) {$\delta$};
\end{tikzpicture}
\]

Given that the horizontal and vertical compositions are also associative,
it is natural to represent composite 2-cells as tilings of rectangles in the plane,
with the appropriate conditions on edges to ensure compatibility between the composed
2-cells. \cite{dawson1993general} have shown that if there are two ways to interpret
such a tiling as a tree of horizontal and vertical compositions, then the resulting
2-cells will be equal. However, there exist tilings which satisfy the local compatibility
conditions but do not arise from the horizontal and vertical compositions. The minimal example
of this is known as the \emph{pinwheel} and is shown in Figure~\ref{fig:pinwheel}.

\begin{figure}
  \centering
\begin{tikzpicture}[scale=0.8]
  \draw (0,0) -- (0,3) -- (3,3) -- (3,0) -- (0,0);
  \draw (0,2) -- (2,2);
  \draw (2,3) -- (2,1);
  \draw (1,1) -- (3,1);
  \draw (1,0) -- (1,2);
\end{tikzpicture}
  \caption{A pinwheel diagram, which cannot be expressed as a binary composite}
  \label{fig:pinwheel}
\end{figure}
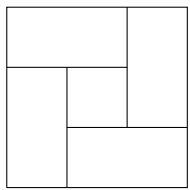
\noindent 
It was then shown by~\cite{dawson1995forbiddensuborder} that this is essentially the only obstacle to
composition of diagrams in double categories.

\section{Free double categories}

Double categories are rich objects and defining them therefore requires
some care. Given horizontal and vertical categories with the same
objects, and a set of generating tiles whose boundaries are chosen
from the horizontal and vertical categories, we want to generate the
free double category on these tiles.

One simple approach to generate such a double category would be to use
its definition as internal category object in $\Cat$, and simply
internalize the definition of a free category on a graph. A graph
object in $\Cat$ is called a \textbf{double graph} and is essentially
a double category without identities and compositions.  Interpreted in
$\Cat$, the construction which defines a free category object from a
graph object does give a double category, but as pointed out by
\cite{dawson2002what} this imposes important restrictions on the
boundaries of the generating tiles: they must be generating morphisms
of the resulting vertical and horizontal categories. Therefore, all
generated composites have a grid-like shape:
\begin{figure}[H]
  \centering
\begin{tikzpicture}[scale=1,yscale=-1]
  \node at (0,0) (p-0-0) {};
  \foreach \x in {1,...,6} {
    \node at (\x,0) (p-\x-0) {};
    \draw[->] ($(p-\x-0)+(-.9,0)$) -- (p-\x-0);
  }
  \foreach \y in {1,...,3} {
    \node at (0,\y) (p-0-\y) {};
    \draw[->] ($(p-0-\y)+(0,-.9)$) -- (p-0-\y);
  }
  \foreach \x in {1,...,6} {
    \foreach \y in {1,...,3} {
          \node at (\x,\y) (p-\x-\y) {};
          \draw[->] ($(p-\x-\y)+(-.9,0)$) -- (p-\x-\y);
          \draw[->] ($(p-\x-\y)+(0,-.9)$) -- (p-\x-\y);
    }
  }
\end{tikzpicture}
\end{figure}
\cite{dawson2002what} propose a more general construction which allows
identities as cell boundaries. To do so they use the notion of
reflexive graph: it is a directed graph with designated loops on each
vertex. One can define the free category generated by a reflexive
graph, where the loops are interpreted as identities. Internalized in
$\Cat$, this gives rise to the notion of \textbf{double reflexive
  graph} which generates a double category.  This makes it possible to
use generators which have identities as boundaries:
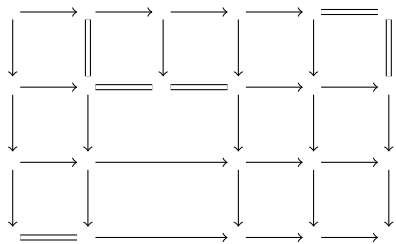
\begin{figure}[H]
  \centering
\begin{tikzpicture}[scale=1,yscale=-1]
  \node at (0,0) (p-0-0) {};
  \foreach \x in {1,...,4} {
    \node at (\x,0) (p-\x-0) {};
    \draw[->] ($(p-\x-0)+(-.9,0)$) -- (p-\x-0);
  }
  \foreach \y in {1,...,3} {
    \node at (0,\y) (p-0-\y) {};
    \draw[->] ($(p-0-\y)+(0,-.9)$) -- (p-0-\y);
  }
  \foreach \x/\y in {1/1,1/2,1/3,2/1,2/2,2/3,3/1,3/2,3/3,4/1,4/2,4/3,5/1,5/2,5/3} {
      \node at (\x,\y) (p-\x-\y) {};
  }
  \foreach \x/\y in {1/1,1/2,4/1,4/2,4/3,5/1,5/2,5/3} {
          \draw[->] ($(p-\x-\y)+(-.9,0)$) -- (p-\x-\y);
  
  }
  \draw[->] ($(p-3-2)+(-1.9,0)$) -- (p-3-2);
  \draw[->] ($(p-3-3)+(-1.9,0)$) -- (p-3-3);
  
  \foreach \x/\y in {1/2,1/3,2/1,3/1,3/2,3/3,4/1,4/2,4/3,5/2,5/3} {
    \draw[->] ($(p-\x-\y)+(0,-.9)$) -- (p-\x-\y);
  }

  \foreach \x/\y in {5/0,2/1,3/1,1/3} {
    \draw[double distance=1.5pt] ($(p-\x-\y)+(-.9,0)$) -- (p-\x-\y);
  }
  \foreach \x/\y in {1/1,5/1} {
    \draw[double distance=1.5pt] ($(p-\x-\y)+(0,-.9)$) -- (p-\x-\y);
  }

\end{tikzpicture}
\end{figure}
As this is still not as general as it could be, \cite{fiore2008model}
introduces the notion of \textbf{double derivation scheme}. A double
derivation scheme is a double graph whose horizontal and vertical
objects form categories. Therefore, generating a double category from
a double derivation scheme makes it possible to use arbitrary
boundaries for the generating cells. The main difference with the
previous approaches is that the notion of double derivation scheme
does not arise by internalizing in $\Cat$ a notion formulated in the
internal language of categories. Moreover a double derivation scheme can also introduce algebraic equations between expressions, quotienting the generated structure accordingly. In our case, no such equations are used, so we give a simpler description of the construction.

\begin{definition}
  A \textbf{double signature} $S = (A,H,V,C)$ is given by:
  \begin{itemize}
  \item a set of objects $A$;
  \item a set of generating horizontal
    morphisms $H$;
  \item a set of generating vertical morphisms $V$;
  \item a set of generating 2-cells $C$.
  \end{itemize}
  Furthermore each $h \in H$ is associated with $\dom h, \cod h \in A$
and similarly for $V$.  This defines free categories $H^*$ and $V^*$.
Each $\alpha \in C$ is associated with compatible vertical and
horizontal domains and codomains $\domh \alpha, \codh \alpha \in H^*$
and $\domv \alpha, \codv \alpha \in V^*$. The required compatibility
is $\dom \domh \alpha = \dom \domv \alpha$ and three other similar equations.
\end{definition}

The set of 2-cells of the double category generated by this data is
generated inductively from the generators in $C$, vertical and
horizontal identities.  From these generators we take the closure by
vertical and horizontal composition of compatible cells: this gives us
the set of 2-cell expressions on the signature. To obtain the set of
2-cells, we quotient by unitality and associativity of the vertical
and horizontal compositions and by the exchange law. Furthermore,
horizontal and vertical identities on identity morphisms (depicted as
empty 2-cells) are equated.  These are precisely the laws of double
categories, hence this defines the free double category $S_d$ on the
given data.

Expressions in double categories
can be drawn as string diagrams~\citep{myers2016string}, and in the
sequel we will use the terms ``expression'' and ``diagram'' interchangeably.
Here is an example of a series of equivalences between expressions of
2-cells, drawn as string diagrams:
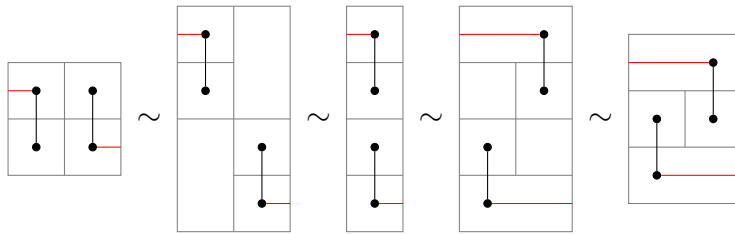
\[
\begin{tikzpicture}[scale=.75]
  \draw[gray] (0,0) rectangle (2,2);
  \draw[gray] (1,0) -- (1,2);
  \draw[gray] (0,1) -- (2,1);
  \node[bnode] at (.5,.5) (a) {};
  \node[bnode] at (.5,1.5) (b) {};
  \node[bnode] at (1.5,.5) (c) {};
  \node[bnode] at (1.5,1.5) (d) {};
  \draw (a) -- (b);
  \draw (c) -- (d);
  \draw[red] (0,1.5) -- (b);
  \draw[red] (2,.5) -- (c);

  \node at (2.5,1) {$\sim$};

  \begin{scope}[xshift=3cm,yshift=-1cm]
    \draw[gray] (0,0) rectangle (2,4);
    \draw[gray] (0,2) -- (2,2);
    \draw[gray] (1,0) -- (1,4);
    \draw[gray] (0,3) -- (1,3);
    \draw[gray] (1,1) -- (2,1);
    \node[bnode] at (.5,2.5) (a) {};
    \node[bnode] at (.5,3.5) (b) {};
    \node[bnode] at (1.5,.5) (c) {};
    \node[bnode] at (1.5,1.5) (d) {};
    \draw (a) -- (b);
    \draw (c) -- (d);
    \draw[red] (0,3.5) -- (b);
    \draw[red] (2,.5) -- (c);

    \node at (2.5,2) {$\sim$};
  \end{scope}

  \begin{scope}[xshift=6cm,yshift=-1cm]
    \draw[gray] (0,0) rectangle (1,4);
    \draw[gray] (0,1) -- (1,1);
    \draw[gray] (0,2) -- (1,2);
    \draw[gray] (0,3) -- (1,3);
    \node[bnode] at (.5,2.5) (a) {};
    \node[bnode] at (.5,3.5) (b) {};
    \node[bnode] at (.5,.5) (c) {};
    \node[bnode] at (.5,1.5) (d) {};
    \draw (a) -- (b);
    \draw (c) -- (d);
    \draw[red] (0,3.5) -- (b);
    \draw[red] (1,.5) -- (c);

    \node at (1.5,2) {$\sim$};
  \end{scope}

  \begin{scope}[xshift=8cm,yshift=-1cm]
    \draw[gray] (0,0) rectangle (2,4);
    \draw[gray] (0,1) -- (2,1);
    \draw[gray] (0,2) -- (2,2);
    \draw[gray] (0,3) -- (2,3);
    \draw[gray] (1,1) -- (1,3);
    \node[bnode] at (1.5,2.5) (a) {};
    \node[bnode] at (1.5,3.5) (b) {};
    \node[bnode] at (.5,.5) (c) {};
    \node[bnode] at (.5,1.5) (d) {};
    \draw (a) -- (b);
    \draw (c) -- (d);
    \draw[red] (0,3.5) -- (b);
    \draw[red] (2,.5) -- (c);

    \node at (2.5,2) {$\sim$};
  \end{scope}

  \begin{scope}[xshift=11cm,yshift=-.5cm]
    \draw[gray] (0,0) rectangle (2,3);
    \draw[gray] (0,1) -- (2,1);
    \draw[gray] (0,2) -- (2,2);
    \draw[gray] (1,1) -- (1,2);
    \node[bnode] at (1.5,1.5) (a) {};
    \node[bnode] at (1.5,2.5) (b) {};
    \node[bnode] at (.5,.5) (c) {};
    \node[bnode] at (.5,1.5) (d) {};
    \draw (a) -- (b);
    \draw (c) -- (d);
    \draw[red] (0,2.5) -- (b);
    \draw[red] (2,.5) -- (c);
  \end{scope}
\end{tikzpicture}
\]
We draw horizontal wires in red, this will help us to to distinguish
them from vertical wires in the next section. We also ommit region
colors as they are irrelevant for equivalences and do not play any
role in the word problem.

Our goal in this work is
to propose an alternate representation for 2-cells, making it possible
to decide whether two expressions of 2-cells are equivalent under these
axioms.

\section{Translation to 2-categories} \label{sec:translation}

Double categories can be seen as a generalization of 2-categories, as
a 2-category is a double category where all vertical morphisms are identities.
Given the inherent duality in double categories, a 2-category can also be seen
as a double category with identity horizontal morphisms.

In this section, we show how a free double category can conversely give
rise to a free 2-category. Our goal is to reuse known algorithms for the
word problem in 2-categories~\citep{delpeuch2018normalization} for
double categories.  To that end we use the string diagram calculus for
2-categories~\citep{selinger2011survey}.

\begin{definition}
  Given a double signature $S = (A,H,V,C)$, we define the 2-category $S_2$
  as the free 2-category generated by:
  \begin{itemize}
  \item objects $a \in A$;
  \item 1-morphisms $h : \dom h \rightarrow \cod h$ for $h \in H$ and $v^\op : \cod v \rightarrow \dom v$ for $v \in H$;
  \item 2-morphisms $\alpha : \domh \alpha \circ (\domv \alpha)^\op \rightarrow (\codv \alpha)^\op \circ \codh \alpha$
  \end{itemize}
\end{definition}
\noindent Note that the vertical generators are reversed in the 2-category,
making it possible to compose the horizontal and vertical domains
together, and similarly for the codomain.

\begin{definition}
  Let $\phi$ be a 2-cell expression in $S_d$. We inductively define its translation $t(\phi)$
  as a morphism in $S_2(\domh \phi \circ (\domv \phi)^\op, (\codv \phi)^\op \circ \codh \phi)$:
  
  \begin{tabular}{l l}
    \begin{tikzpicture}
  \draw[gray] (0,0) rectangle (1,1);
  \node[circle,draw,inner sep=2pt] at (.5,.5) (alpha) {$\alpha$};
  \draw[red] (0,.5) -- (alpha) -- (1,.5);
  \draw (.5,0) -- (alpha) -- (.5,1);

  \node at (3,.5) {$\mapsto$};

  \begin{scope}[xshift=4cm]
  \draw[gray] (0,0) rectangle (1,1);
  \node[circle,draw,inner sep=2pt] at (.5,.5) (alpha) {$\alpha$};
  \draw[red] (.25,1) -- (alpha) -- (.75,0);
  \draw (.25,0) -- (alpha) -- (.75,1);
  \end{scope}
\end{tikzpicture} & \text{generator} \\
    \begin{tikzpicture}
  \draw[gray] (0,0) rectangle (2,1);
  \node[circle,draw,inner sep=1.5pt] at (.5,.5) (mu) {$\mu$};
  \node[circle,draw,inner sep=2pt] at (1.5,.5) (nu) {$\nu$};
  \draw[red] (0,.5) -- (mu) -- (nu) -- (2,.5);
  \draw (.5,0) -- (mu) -- (.5,1);
  \draw (1.5,0) -- (nu) -- (1.5,1);

  \node at (3,.5) {$\mapsto$};

  \begin{scope}[xshift=4cm,yshift=-.5cm]
  \draw[gray] (0,0) rectangle (2,2);
  \node[circle,draw,inner sep=1.5pt] at (.66,1.5) (mu) {$\mu$};
  \node[circle,draw,inner sep=2pt] at (1.33,.5) (nu) {$\nu$};
  \draw[red] (.25,2) -- (mu) -- (nu) -- (1.75,0);
  \draw (.25,0) -- (mu) -- (.8,2);
  \draw (1.2,0) -- (nu) -- (1.75,2);
  \end{scope}
\end{tikzpicture} & \text{horizontal composition} \\
   \begin{tikzpicture}
  \begin{scope}[xshift=1cm,yshift=-.5cm,rotate=90]
  \draw[gray] (0,0) rectangle (2,1);
  \node[circle,draw,inner sep=2pt] at (.5,.5) (mu) {$\nu$};
  \node[circle,draw,inner sep=1.5pt] at (1.5,.5) (nu) {$\mu$};
  \draw (0,.5) -- (mu) -- (nu) -- (2,.5);
  \draw[red] (.5,0) -- (mu) -- (.5,1);
  \draw[red] (1.5,0) -- (nu) -- (1.5,1);
  \end{scope}

  \node at (3,.5) {$\mapsto$};

  \begin{scope}[xshift=7cm,yshift=-.5cm,xshift=-1cm,xscale=-1]
  \draw[gray] (0,0) rectangle (2,2);
  \node[circle,draw,inner sep=1.5pt] at (.66,1.5) (mu) {$\mu$};
  \node[circle,draw,inner sep=2pt] at (1.33,.5) (nu) {$\nu$};
  \draw (.25,2) -- (mu) -- (nu) -- (1.75,0);
  \draw[red] (.25,0) -- (mu) -- (.8,2);
  \draw[red] (1.2,0) -- (nu) -- (1.75,2);
  \end{scope}
\end{tikzpicture} & \text{vertical composition} \\
    \begin{tikzpicture}
  \draw[gray] (0,0) rectangle (1,1);

  \draw[red] (0,.5) -- (1,.5);

  \node at (3,.5) {$\mapsto$};

  \begin{scope}[xshift=4cm]
  \draw[gray] (0,0) rectangle (1,1);
  \draw[red] (.5,1) -- (.5,0);
  \end{scope}
\end{tikzpicture}
 & \text{horizontal identity} \\
   \begin{tikzpicture}
  \draw[gray] (0,0) rectangle (1,1);

  \draw (.5,0) -- (.5,1);

  \node at (3,.5) {$\mapsto$};

  \begin{scope}[xshift=4cm]
  \draw[gray] (0,0) rectangle (1,1);
  \draw (.5,0) -- (.5,1);
  \end{scope}
\end{tikzpicture} & \text{vertical identity}
  \end{tabular}
\end{definition}

\begin{lemma}
  The translation $t$ respects the axioms of double categories, i.e. it extends
  to a map from 2-cells in $S_d$ to 2-cells in $S_2$.
\end{lemma}

\begin{proof}
  One can check that unitality and associativity are respected. The exchange law
  in double categories translates to the exchange law in 2-categories:
  
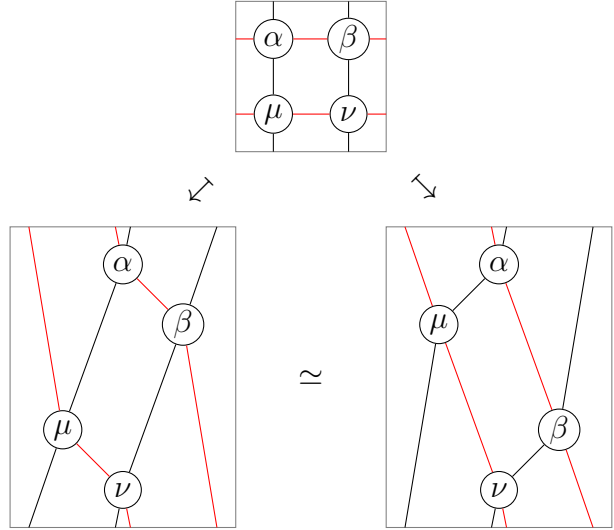
\begin{center}
\begin{tikzpicture}
  \draw[gray] (0,0) rectangle (2,2);
  \node[circle,draw,inner sep=1.5pt] at (.5,.5) (mu) {$\mu$};
  \node[circle,draw,inner sep=2pt] at (1.5,.5) (nu) {$\nu$};
  \node[circle,draw,inner sep=2pt] at (.5,1.5) (alpha) {$\alpha$};
  \node[circle,draw,inner sep=1pt] at (1.5,1.5) (beta) {$\beta$};

  \draw (.5,0) -- (mu) -- (alpha) -- (.5,2);
  \draw (1.5,0) -- (nu) -- (beta) -- (1.5,2);
  \draw[red] (0,.5) -- (mu) -- (nu) -- (2,.5);
  \draw[red] (0,1.5) -- (alpha) -- (beta) -- (2,1.5);

  \node[rotate=-45] at (2.5,-.5) {$\mapsto$};
  \node[rotate=-135] at (-.5,-.5) {$\mapsto$};
  \node at (1,-3) {$\simeq$};

  \begin{scope}[xshift=2cm,yshift=-5cm]
    \draw[gray] (0,0) rectangle (3,4);
    \node[circle,draw,inner sep=1.5pt] at (.7,2.7) (mu) {$\mu$};
    \node[circle,draw,inner sep=2pt] at (1.5,3.5) (alpha) {$\alpha$};
    \node[circle,draw,inner sep=2pt] at (1.5,0.5) (nu) {$\nu$};
    \node[circle,draw,inner sep=1pt] at (2.3,1.3) (beta) {$\beta$};
    \draw[red] (.25,4) -- (mu) -- (nu) -- (1.6,0);
    \draw[red] (1.4,4) -- (alpha) -- (beta) -- (2.75,0);
    \draw (.25,0) -- (mu) -- (alpha) -- (1.6,4);
    \draw (1.4,0) -- (nu) -- (beta) -- (2.75,4);
  \end{scope}

  \begin{scope}[xshift=-3cm,yshift=-5cm]
    \draw[gray] (0,0) rectangle (3,4);
    \node[circle,draw,inner sep=1.5pt] at (.7,1.3) (mu) {$\mu$};
    \node[circle,draw,inner sep=2pt] at (1.5,3.5) (alpha) {$\alpha$};
    \node[circle,draw,inner sep=2pt] at (1.5,0.5) (nu) {$\nu$};
    \node[circle,draw,inner sep=1pt] at (2.3,2.7) (beta) {$\beta$};
    \draw[red] (.25,4) -- (mu) -- (nu) -- (1.6,0);
    \draw[red] (1.4,4) -- (alpha) -- (beta) -- (2.75,0);
    \draw (.25,0) -- (mu) -- (alpha) -- (1.6,4);
    \draw (1.4,0) -- (nu) -- (beta) -- (2.75,4);
  \end{scope}
\end{tikzpicture}
  \end{center}
\end{proof}

Our goal is to show the converse: if the translations of two expressions in $S_d$ are
equivalent as morphisms in $S_2$, then so are their antecedents in $S_d$. To do so, we need to
construct a reverse translation, from diagrams in the free 2-category
to diagrams in the free double category.

\section{Partial tilings}

To provide an inverse to the translation $t$, let us first introduce a necessary condition
on a diagram in $S_2$ to be in the image of $t$.

\begin{definition}
  A diagram $\phi \in S_2$ is \textbf{admissible} if its domain is of
  the form $v^{\op} ; h$ and its codomain is of the form $h'; v'^{\op}$.
\end{definition}

For all $\psi \in S_d$, $t(\psi)$ is admissible. Conversely, for all admissible $\phi \in S_2$,
we want to construct a corresponding tiling.
To do so, we introduce the notion of partial tiling as
an incomplete diagram in the free double category.

\begin{definition} \label{defi:partial-tiling}
  Let $n \geq 1$, and $h, h_1, \dots, h_n \in H^*$ and $v, v_1, \dots, v_n \in V^*$.
  Assume that $h_i$ is not an identity for $i > 1$ and $v_i$ is not an identity for $i < n$.
  
  A \textbf{partial tiling} of type $h, v \rightarrow h_1, v_1, \dots, h_n, v_n$ is
  a subdivision of the following shape into rectangles:
  
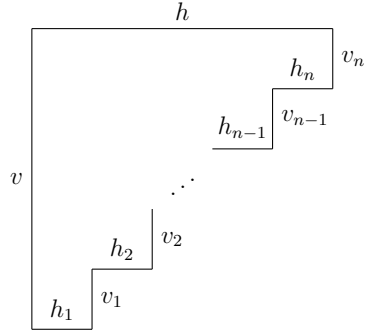
\[
\begin{tikzpicture}[scale=0.8, every node/.style={scale=0.8}]
  \draw (0,0) edge node[left] {$v$} (0,5);
  \draw (0,5) edge node[above] {$h$} (5,5);
  \draw (5,5) edge node[right] {$v_n$} (5,4);
  \draw (5,4) edge node[above] {$h_n$} (4,4);
  \draw (4,4) edge node[right] {$v_{n-1}$} (4,3);
  \draw (4,3) edge node[above] {$h_{n-1}$} (3,3);

  \draw (0,0) edge node[above] {$h_1$} (1,0);
  \draw (1,0) edge node[right] {$v_1$} (1,1);
  \draw (1,1) edge node[above] {$h_2$} (2,1);
  \draw (2,1) edge node[right] {$v_2$} (2,2);

  \node at (2.5,2.5) {$\iddots$};
\end{tikzpicture}
  \]
  Each of the rectangles in the subdivision is attributed a generator
  $\alpha \in C$ or a vertical or horizontal identity, such that the
  horizontal and vertical domains and codomains match on each edge.
\end{definition}

We think of a partial tiling as some upper-left corner of a 2-cell in a double category. We will therefore draw partial tilings just like string diagrams for double categories, as in Figure~\ref{fig:example-partial-tilings}.

\begin{definition}
  Two partial tilings are \textbf{equivalent} when they can be related by a series
  of applications of these rules (where $\alpha$ can be an identity itself):

  
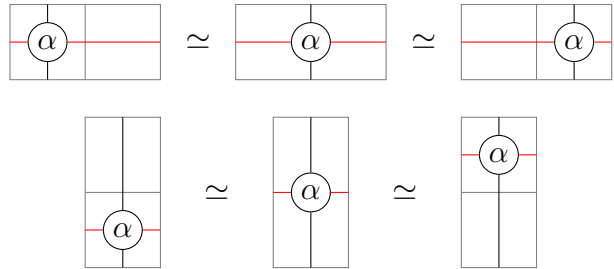
\begin{center}
\begin{tikzpicture}
  \draw[gray] (0,0) rectangle (2,1);
  \draw[gray] (1,0) -- (1,1);
  \node[circle,draw,inner sep=2pt] at (.5,.5) (alpha) {$\alpha$};
  \draw[red] (0,.5) -- (alpha) -- (2,.5);
  \draw (.5,0) -- (alpha) -- (.5,1);

  \node at (2.5,.5) {$\simeq$};

  \begin{scope}[xshift=3cm]
  \draw[gray] (0,0) rectangle (2,1);
  \node[circle,draw,inner sep=2pt] at (1,.5) (alpha) {$\alpha$};
  \draw[red] (0,.5) -- (alpha) -- (2,.5);
  \draw (1,0) -- (alpha) -- (1,1);
  \end{scope}

  \node at (5.5,.5) {$\simeq$};
  \begin{scope}[xshift=8cm,xscale=-1]
      \draw[gray] (0,0) rectangle (2,1);
  \draw[gray] (1,0) -- (1,1);
  \node[circle,draw,inner sep=2pt] at (.5,.5) (alpha) {$\alpha$};
  \draw[red] (0,.5) -- (alpha) -- (2,.5);
  \draw (.5,0) -- (alpha) -- (.5,1);
  \end{scope}

  \begin{scope}[yshift=-2.5cm]
  \begin{scope}[xshift=2cm,rotate=90]
  \draw[gray] (0,0) rectangle (2,1);
  \draw[gray] (1,0) -- (1,1);
  \node[circle,draw,inner sep=2pt] at (.5,.5) (alpha) {$\alpha$};
  \draw (0,.5) -- (alpha) -- (2,.5);
  \draw[red] (.5,0) -- (alpha) -- (.5,1);
  \end{scope}

  \node at (2.75,1) {$\simeq$};

  \begin{scope}[xshift=4.5cm,rotate=90]
  \draw[gray] (0,0) rectangle (2,1);
  \node[circle,draw,inner sep=2pt] at (1,.5) (alpha) {$\alpha$};
  \draw (0,.5) -- (alpha) -- (2,.5);
  \draw[red] (1,0) -- (alpha) -- (1,1);
  \end{scope}

  \node at (5.25,1) {$\simeq$};
  \begin{scope}[xshift=7cm,yshift=2cm,rotate=90,xscale=-1]
      \draw[gray] (0,0) rectangle (2,1);
  \draw[gray] (1,0) -- (1,1);
  \node[circle,draw,inner sep=2pt] at (.5,.5) (alpha) {$\alpha$};
  \draw (0,.5) -- (alpha) -- (2,.5);
  \draw[red] (.5,0) -- (alpha) -- (.5,1);
  \end{scope}
  \end{scope}

\end{tikzpicture}
  \end{center}
  
  \noindent as well as continuous translations of horizontal
  or vertical boundaries in the subdivision. We denote this equivalence by $\simeq$.
\end{definition}
\noindent For instance, the two partial tilings in Figure~\ref{fig:example-partial-tilings} are equivalent.

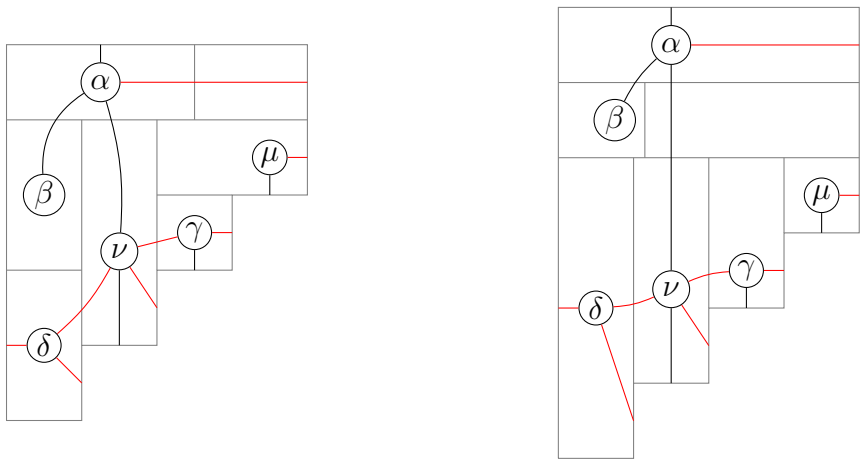
\begin{figure}
  \centering
  \begin{subfigure}{0.45\textwidth}
    \centering
\begin{tikzpicture}
\begin{scope}[every path/.style={gray}]
\draw (0,0) -- (0,5) -- (4,5) -- (4,3) -- (3,3) -- (3,2) -- (2,2) -- (2,1) -- (1,1) -- (1,0) -- (0,0);
\draw (0,4) -- (4,4);
\draw (1,4) -- (1,1);
\draw (0,2) -- (1,2);
\draw (2,2) -- (2,4);
\draw (2,3) -- (3,3);
\draw (2.5,4) -- (2.5,5);
\end{scope}

\node[circle,draw,inner sep=2pt] at (1.25,4.5) (alpha) {$\alpha$};
\node[circle,draw,inner sep=1pt] at (.5,3) (beta) {$\beta$};
\node[circle,draw,inner sep=1pt] at  (.5,1) (delta) {$\delta$};
\node[circle,draw,inner sep=2pt] at (1.5,2.25) (nu) {$\nu$};
\node[circle,draw,inner sep=1pt] at (3.5,3.5) (mu) {$\mu$};
\node[circle,draw,inner sep=1pt] at (2.5,2.5) (gamma) {$\gamma$};

\begin{scope}[every path/.style={red}]
  \draw (alpha) -- (4,4.5);
  \draw (delta) edge[bend right=10] (nu);
  \draw (delta) -- (1,.5);
  \draw (mu) -- (4,3.5);
  \draw (nu) -- (2,1.5);
  \draw (nu) -- (gamma);
  \draw (gamma) -- (3,2.5);
  \draw (0,1) -- (delta);
\end{scope}

\draw (beta) edge[bend left] (alpha);
\draw (alpha) -- ($(alpha)+(0,.5)$);
\draw (alpha) edge[bend left=10] (nu);
\draw (nu) -- (1.5,1);
\draw (gamma) -- (2.5,2);
\draw (mu) -- (3.5,3);

\end{tikzpicture}
  \end{subfigure}
  \begin{subfigure}{0.45\textwidth}
    \centering
\begin{tikzpicture}
\begin{scope}[every path/.style={gray}]
\draw (0,0) -- (0,6) -- (4,6) -- (4,3) -- (3,3) -- (3,2) -- (2,2) -- (2,1) -- (1,1) -- (1,0) -- (0,0);
\draw (0,4) -- (4,4);
\draw (1,4) -- (1,1);
\draw (2,2) -- (2,4);
\draw (1.15,4) -- (1.15,5);
\draw (0,5) -- (4,5);
\draw (3,4) -- (3,3);
\end{scope}

\node[circle,draw,inner sep=2pt] at (1.5,5.5) (alpha) {$\alpha$};
\node[circle,draw,inner sep=1pt] at (.75,4.5) (beta) {$\beta$};
\node[circle,draw,inner sep=1pt] at  (.5,2) (delta) {$\delta$};
\node[circle,draw,inner sep=2pt] at (1.5,2.25) (nu) {$\nu$};
\node[circle,draw,inner sep=1pt] at (3.5,3.5) (mu) {$\mu$};
\node[circle,draw,inner sep=1pt] at (2.5,2.5) (gamma) {$\gamma$};

\begin{scope}[every path/.style={red}]
  \draw (alpha) -- (4,5.5);
  \draw (delta) edge[bend right=10] (nu);
  \draw (delta) -- (1,.5);
  \draw (mu) -- (4,3.5);
  \draw (nu) -- (2,1.5);
  \draw (nu) edge[bend left=10] (gamma);
  \draw (gamma) -- (3,2.5);
  \draw (0,2) -- (delta);
\end{scope}

\draw (beta) edge[bend left=10] (alpha);
\draw (alpha) -- ($(alpha)+(0,.5)$);
\draw (alpha) edge (nu);
\draw (nu) -- (1.5,1);
\draw (gamma) -- (2.5,2);
\draw (mu) -- (3.5,3);

\end{tikzpicture}
  \end{subfigure}
  \caption{Examples of partial tilings}
  \label{fig:example-partial-tilings}
\end{figure}

In the special case where $n = 2$, $h_1$ and $v_2$ are identities and $h = h_2$, $v = v_2$, and assuming $\dom h = \dom v$, there is an \textbf{empty partial tiling} of type $h, v \rightarrow 1_{\cod v}, v, h, 1_{\cod h}$:

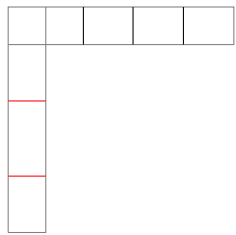
\[
\begin{tikzpicture}
  \draw[gray] (0,0) -- (0,3) -- (3,3) -- (3,2.5) -- (.5,2.5) -- (.5,0) -- (0,0);
  \draw[gray] (.5,2.5) -- (0,2.5);
  \draw[gray] (.5,2.5) -- (.5,3);
  \draw[red] (0,.75) -- (.5,.75);
  \draw[red] (0,1.75) -- (.5,1.75);
  \draw (1,3) -- (1,2.5);
  \draw (1.66,3) -- (1.66,2.5);
  \draw (2.33,3) -- (2.33,2.5);
\end{tikzpicture}
\]

Another special case are partial tilings of type $h, v \rightarrow h_1, v_1$ which have a rectangular shape. In this case, we can interpret them as 2-cells, but only if they are binary composable.

\begin{lemma} \label{lemma:partial-tiling-complete}
  A partial tiling of type $h, v \rightarrow h_1, v_1$ is \textbf{binary composable}
  if it can be obtained by repeated application of the horizontal and vertical composition
  from generators. In this case, it represents a 2-cell in $S_d$,
  and its meaning is invariant under equivalence of partial tilings.
\end{lemma}

\begin{proof}
  If a diagram is binary composable then by the general associativity result
  of \cite{dawson1993general}, it can be interpreted as a 2-cell in $S_d$ which
  does not depend on the order of composition chosen.
  Then, equivalences of partial tilings correspond to unitality of identities
  when interpreted in a binary composable diagram, so the 2-cell is invariant
  under these equivalences.
\end{proof}

\begin{definition}
 Let $h$ be an horizontal morphism in $S_d$. It can be uniquely decomposed as
a composition of generators $h = h_1 \circ \dots \circ h_k$. We define the \emph{length} of $h$ as $|h| = k$. Let $0 \leq i < k$ and $1 \leq j < k$. We say that $h' = h_{i+1} \circ \dots \circ h_j$ is a \emph{factor at index} $i$ of $h$.
Similar notions are defined for vertical morphisms.
\end{definition}

\noindent For instance, the factors at index $0$ of a morphism $h$ are its prefixes.

\begin{definition}
  Let $m$ be a partial tiling of type $h, v \rightarrow h_1, v_1, \dots, h_n, v_n$
  and let $\alpha : h', v' \rightarrow h'', v''$ be a generator. A \textbf{gluing position}
  of $\alpha$ on $m$ is one of the following:
  \begin{itemize}
  \item if $h'$ is a prefix of $h_1$, then $(0,0,0)$ is a gluing position;
  \item if $h'$ is a factor of $h_1$ at index $i > 0$ and $v'$ is an identity, then $(0,i,0)$ is a gluing position;
  \item if $v'$ is a prefix of $v_n$, then $(n,0,0)$ is a gluing position;
  \item if $v'$ is a factor of $v_n$ at index $i > 0$ and $h'$ is an identity, then $(n,0,i)$ is a gluing position;
  \end{itemize}
  Furthermore, for all $1 \leq k < n$:
  \begin{itemize}
     \item if $v'$ is a prefix of $v_k$ and $h'$ is a prefix of $h_{k+1}$,        then $(k,0,0)$ is a gluing position;
     \item if $v'$ is a factor of $v_k$ at index $i$ and $h'$ is an identity, then $(k,0,i)$ is a gluing position;
     \item if $h'$ is a factor of $h_{k+1}$ at index $i$ and $v'$ is an identity, then $(k,i,0)$ is a gluing position.     
  \end{itemize}
  For each gluing position $p$ we define the \textbf{gluing} of $\alpha$ to $m$, denoted by $m \star_p \alpha$, by the partial tiling obtained by adjoining $\alpha$ at the designated position,
  and adding any necessary identity to satisfy the condition of Definition~\ref{defi:partial-tiling}.
\end{definition}

Figure~\ref{fig:gluing-positions} shows all the possible gluing positions. Figure~\ref{fig:gluing} shows how a generator can be glued at
multiple positions on a partial tiling and how identities can be used
to ensure that the resulting arrangement has non-identity inner
boundaries.

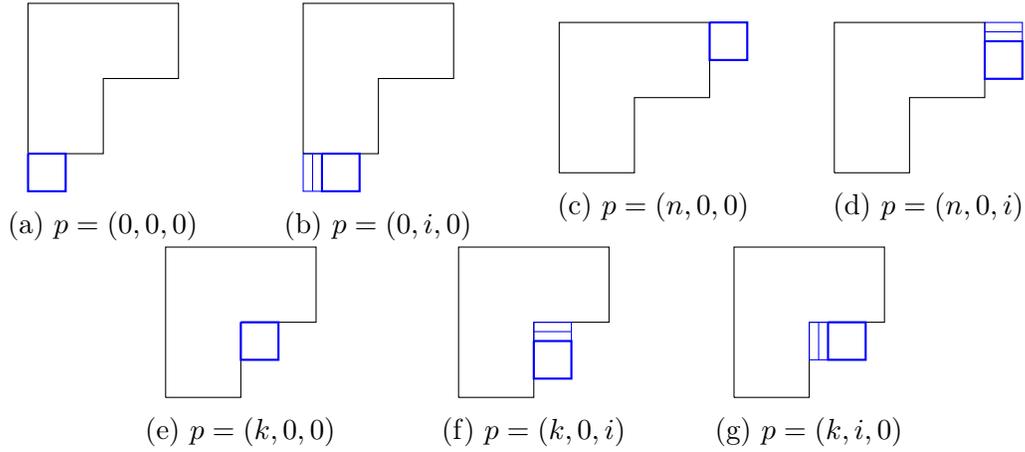
\begin{figure}
  \centering
\begin{subfigure}{0.22\textwidth}
  \centering
  \begin{tikzpicture}
    \draw (0,0) -- (1,0) -- (1,1) -- (2,1) -- (2,2) -- (0,2) -- (0,0);
    \draw[blue,thick] (0,0) -- (.5,0) -- (.5,-.5) -- (0,-.5) -- (0,0);
  \end{tikzpicture}
  \caption{$p = (0,0,0)$}
\end{subfigure}
\begin{subfigure}{0.22\textwidth}
  \centering
  \begin{tikzpicture}
    \draw (0,0) -- (1,0) -- (1,1) -- (2,1) -- (2,2) -- (0,2) -- (0,0);
    \draw[blue] (0,0) -- (.25,0) -- (.25,-.5) -- (0,-.5) -- (0,0);
    \draw[blue] (.125,0) -- (.125,-.5);
    \begin{scope}[xshift=.25cm]
      \draw[blue,thick] (0,0) -- (.5,0) -- (.5,-.5) -- (0,-.5) -- (0,0);
    \end{scope}
  \end{tikzpicture}
  \caption{$p = (0,i,0)$}
\end{subfigure}
\begin{subfigure}{0.22\textwidth}
  \centering
  \begin{tikzpicture}
    \draw (0,0) -- (1,0) -- (1,1) -- (2,1) -- (2,2) -- (0,2) -- (0,0);
    \begin{scope}[xshift=2cm,yshift=2cm]
      \draw[blue,thick] (0,0) -- (.5,0) -- (.5,-.5) -- (0,-.5) -- (0,0);
    \end{scope}
  \end{tikzpicture}
  \caption{$p = (n,0,0)$}
\end{subfigure}
\begin{subfigure}{0.22\textwidth}
  \centering
  \begin{tikzpicture}
    \draw (0,0) -- (1,0) -- (1,1) -- (2,1) -- (2,2) -- (0,2) -- (0,0);
    \draw[blue] (2,2) -- (2,1.75) -- (2.5,1.75) -- (2.5,2) -- (2,2);
    \draw[blue] (2,1.875) -- (2.5,1.875);
    \begin{scope}[yshift=1.75cm,xshift=2cm]
      \draw[blue,thick] (0,0) -- (.5,0) -- (.5,-.5) -- (0,-.5) -- (0,0);
    \end{scope}
  \end{tikzpicture}
  \caption{$p = (n,0,i)$}
\end{subfigure}
\begin{subfigure}{0.25\textwidth}
  \centering
  \begin{tikzpicture}
    \draw (0,0) -- (1,0) -- (1,1) -- (2,1) -- (2,2) -- (0,2) -- (0,0);
    \begin{scope}[xshift=1cm,yshift=1cm]
      \draw[blue,thick] (0,0) -- (.5,0) -- (.5,-.5) -- (0,-.5) -- (0,0);
    \end{scope}
  \end{tikzpicture}
  \caption{$p = (k,0,0)$}
\end{subfigure}
\begin{subfigure}{0.22\textwidth}
  \centering
  \begin{tikzpicture}
    \draw (0,0) -- (1,0) -- (1,1) -- (2,1) -- (2,2) -- (0,2) -- (0,0);
    \begin{scope}[xshift=-1cm,yshift=-1cm]
    \draw[blue] (2,2) -- (2,1.75) -- (2.5,1.75) -- (2.5,2) -- (2,2);
    \draw[blue] (2,1.875) -- (2.5,1.875);
    \begin{scope}[yshift=1.75cm,xshift=2cm]
      \draw[blue,thick] (0,0) -- (.5,0) -- (.5,-.5) -- (0,-.5) -- (0,0);
    \end{scope}
    \end{scope}
  \end{tikzpicture}
  \caption{$p = (k,0,i)$}
\end{subfigure}
\begin{subfigure}{0.22\textwidth}
  \centering
  \begin{tikzpicture}
    \draw (0,0) -- (1,0) -- (1,1) -- (2,1) -- (2,2) -- (0,2) -- (0,0);
    \begin{scope}[xshift=1cm,yshift=1cm]
      \draw[blue] (0,0) -- (.25,0) -- (.25,-.5) -- (0,-.5) -- (0,0);
      \draw[blue] (.125,0) -- (.125,-.5);
      \begin{scope}[xshift=.25cm]
        \draw[blue,thick] (0,0) -- (.5,0) -- (.5,-.5) -- (0,-.5) -- (0,0);
      \end{scope}
    \end{scope}
  \end{tikzpicture}
  \caption{$p = (k,i,0)$}
\end{subfigure}
  \caption{Possible gluing positions. When the second or third component of the position is not null, an identity cell is added.}
  \label{fig:gluing-positions}
\end{figure}
\begin{figure}
  \centering
\begin{tikzpicture}[every node/.style={scale=.8}]
  \node at (-.5,1) {$m =$};
  
  \draw[gray] (0,0) -- (0,2) -- (2,2) -- (2,1) -- (1,1) -- (1,0) -- (0,0);
  \draw[gray] (0,1) -- (1,1);

  \node[circle,draw,inner sep=1pt] at (.5,.5) (beta) {$\beta$};
  \node[circle,draw,inner sep=1.5pt] at (1,1.5) (delta) {$\delta$};

  \draw[red] (beta) edge[bend right=10] (1,.33);
  \draw[red] (beta) edge[bend left=10] (1,.66);
  \draw[red] (delta) edge (2,1.5);
  \draw[red] (0,1.5) edge (delta);
  \draw (delta) edge[bend right=10] (beta);
  \draw (delta) edge[bend left=10] (1.33,1);
  \draw (delta) edge[bend left=10] (1.66,1);

  \begin{scope}[xshift=5cm,yshift=.5cm]
    \node at (-.5,.5) {$\alpha =$};
    \draw[gray] (0,0) -- (0,1) -- (1,1) -- (1,0) -- (0,0);
    \node[circle,draw,inner sep=1pt] at (.5,.5) (alpha) {$\alpha$};
    \draw[red] (0,.5) -- (alpha);
    \draw (.5,1) -- (alpha);
  \end{scope}

  \begin{scope}[xshift=0cm,yshift=-3cm]
      \node at (-.9,1) {$m \star_{(1,0,0)} \alpha =$};
  
  \draw[gray] (0,0) -- (0,2) -- (2,2) -- (2,1) -- (1,1) -- (1,0) -- (0,0);
  \draw[gray] (0,1) -- (1,1);
  \draw[gray] (1,.3) -- (1.7,.3) -- (1.7,1);
  \draw[gray] (1,0) -- (1.7,0) -- (1.7,0.3);
  \draw[gray] (2,1) -- (2,.3) -- (1.7,.3);

  \node[circle,draw,inner sep=1pt] at (.5,.5) (beta) {$\beta$};
  \node[circle,draw,inner sep=1pt] at (1.4,.6) (alpha) {$\alpha$};
  \node[circle,draw,inner sep=1.5pt] at (1,1.5) (delta) {$\delta$};

  \draw[red] (beta) edge[bend right=10] (1,.15);
  \draw[red] (1,.15) -- (1.7,.15);
  \draw[red] (beta) edge[bend left=10] (alpha);
  \draw[red] (delta) edge (2,1.5);
  \draw[red] (0,1.5) edge (delta);
  \draw (delta) edge[bend right=10] (beta);
  \draw (delta) edge[bend left=10] (alpha);
  \draw (delta) edge[bend left=10] (1.85,1);
  \draw (1.85,1) -- (1.85,0.3);
  \end{scope}

  \begin{scope}[xshift=5cm,yshift=-3cm]
  \node at (-1,1) {$m \star_{(2,0,0)} \alpha =$};
  
  \draw[gray] (0,0) -- (0,2) -- (2,2) -- (2,1) -- (1,1) -- (1,0) -- (0,0);
  \draw[gray] (0,1) -- (1,1);
  \draw[gray] (2,1) -- (3,1) -- (3,2) -- (2,2);

  \node[circle,draw,inner sep=1pt] at (.5,.5) (beta) {$\beta$};
  \node[circle,draw,inner sep=1.5pt] at (1,1.5) (delta) {$\delta$};

  \draw[red] (beta) edge[bend right=10] (1,.33);
  \draw[red] (beta) edge[bend left=10] (1,.66);
  \draw[red] (delta) edge (2,1.5);
  \draw[red] (0,1.5) edge (delta);
  \draw (delta) edge[bend right=10] (beta);
  \draw (delta) edge[bend left=10] (1.33,1);
  \draw (delta) edge[bend left=10] (1.66,1);

  \begin{scope}[xshift=2cm,yshift=1cm]
      \node[circle,draw,inner sep=1pt] at (.5,.5) (alpha) {$\alpha$};
    \draw[red] (0,.5) -- (alpha);
    \draw (.5,1) -- (alpha);
  \end{scope}
  \end{scope}
\end{tikzpicture}
  \caption{Example of gluings of a generator on a partial tiling}
  \label{fig:gluing}
\end{figure}
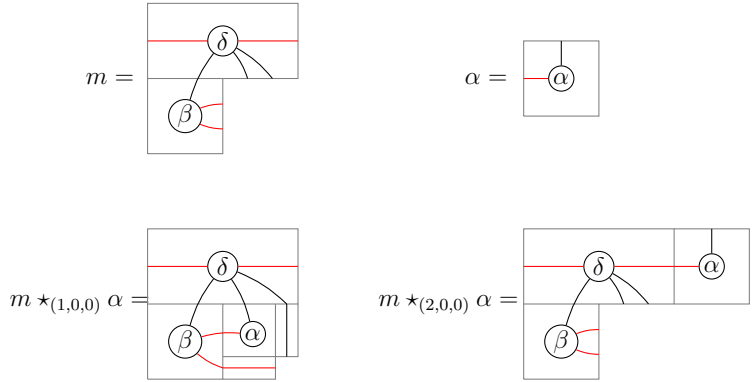

\begin{definition}
  Let $\phi$ be a diagram in the free 2-category $S_2$.
  We assume that the domain of $\phi$ is of the form $v^{\op} ; h$ with $v$ and $h$ paths of generators
  from $S$.
  
  Let $l \in \mathbb{N}$ be a level in $\phi$ and $h_1; v_1^\op ; \dots ; h_n; v_n^\op$ be the
  type of the diagram at this height, where $v_1$ and $h_n$ can possibly be identities unlike the others
  elements of the sequence.

  We associate to this data a partial tiling $p_k(\phi) : h, v \rightarrow h_1, v_1, \dots, h_n, v_n$, by induction
  on $k$.
  If $k = 0$, $p_k(\phi)$ is the empty partial tiling of type
  $v, h \rightarrow 1, v, h, 1$.  Otherwise, let $\alpha$ be the
  generator between levels $k-1$ and $k$. We define $p_k(\phi)$ as the
  gluing of $\alpha$ on $p_{k-1}(\phi)$ at the position indicated by
  the connection of $\alpha$ to the level $k-1$ of $\phi$.

  Finally we define $p(\phi)$ as $p_f(\phi)$ for $f$ the final level
  of $\phi$.
\end{definition}

\noindent The construction relies on the following two lemmas:
\begin{lemma}
  Let $\alpha$ be the generator between slices $k$ and $k+1$ in $\phi$, a diagram in $S_2$.
  If $\alpha$ has at least one input wire, this determines a unique gluing position $g$
  of $\alpha$ on $p_k(\phi)$.
\end{lemma}

\begin{proof}
  Each wire crossing level $k$ in $\phi$ corresponds to an open wire on the boundary of $p_k(\phi)$, either in a vertical or horizontal boundary depending on the colour of the wire.

  Let $v^\op ; h$ be the domain of $\alpha$.

  By assumption, at least one of $v, h$ is not an identity. Assume first that $v$ is not an identity.
  As no horizontal inner boundaries of $p_k(\phi)$ are identities, as required by Definition~\ref{defi:partial-tiling}, any contiguous sequence of red wires in $\phi$ corresponds to a contiguous sequence of wires on some vertical boundary $v_i$ of $p_k(\phi)$.

  Let $j$
  be such that $v$ is a factor at index $j$ in $v_i$.
  One can then check that $(i,0,j)$ is a valid
  gluing position for $\alpha$ on $p_k(\phi)$.

  Similarly, if $h$ is not an identity, then the corresponding wires in $\phi$ determine a unique occurrence of $h$ in a vertical boundary $h_i$ of $p_k(\phi)$,
  and by denoting by $j$ the index of $h$ in $h_i$,
  this determines the gluing position $(i-1,j,0)$.
\end{proof}

\begin{lemma} \label{lemma:gluing-no-input}
  Let again $\alpha$ be the generator between slices $h$ and $h+1$ in $\phi$, a diagram in $S_2$.
  If $\alpha$ has no input wire, meaning that its domain is the identity, then this determines either one or
  two gluing positions of $\alpha$ on $p_h(\phi)$. If there are two such positions $l, l'$ then $p_h(\phi) \star_{l} \alpha \simeq p_h(\phi) \star_{l'} \alpha$.
\end{lemma}

\begin{proof}
  Let $h_1 ; v_1^\op ; \dots ; h_n ; v_n^\op$ be the type of the diagram at height $h$.
  Again, each wire in this sequence corresponds to an open wire on the boundary of $p_k(\phi)$.
  The wires passing to the left of $\alpha$ in $\phi$ determine a position in this sequence where
  the generator $\alpha$ is inserted. The gluing positions depend on this position.

  If $\alpha$ is bordered by two horizontal wires on each side (respectively two vertical wires),
  this determines a unique gluing position $(k, i, 0)$ (respectively $(k, 0, i)$) as in the previous lemma. Similarly,
  if $\alpha$ neighbours a vertical wire on its left and a horizontal wire on its right, this
  determines a unique gluing position $(k,0,0)$ as in the previous lemma.

  The remaining cases are when $\alpha$ neighbours a horizontal wire
  on its left and a vertical wire on its right, when $\alpha$ does not
  have any wire on its left and a vertical one on its right, when it
  has a horizontal wire on its left and none on its right, or when
  there are no wires neither on the left or the right of $\alpha$.  In
  this case this determines two gluing positions $l = (k,i,0)$ and $l'
  = (k+1,0,j)$, and Figure~\ref{fig:equiv-gluing-positions} shows how
  $p_h(\phi) \star_l \alpha \simeq p_h(\phi) \star_{l'} \alpha$ in
  this case.
\end{proof}

\begin{figure}
\begin{tikzpicture}
  \draw[gray] (0,0) -- (2,0) -- (2,2) -- (0,2) -- (0,0);
  \draw[gray] (0,1) -- (2,1);
  \draw[gray] (1,0) -- (1,1);

  \node[circle,draw,inner sep=1pt] at (.5,1.5) (beta) {$\beta$};
  \node[circle,draw,inner sep=2pt] at (1.5,.5) (alpha) {$\alpha$};
  \draw[red] (beta) -- (2,1.5);
  \draw (beta) -- (.5,0);
  \draw[red] (alpha) -- (2,.5);
  \draw (alpha) -- (1.5,0);

  \node at (3,1) {$\simeq$};

  \begin{scope}[xshift=4cm]
    \draw[gray] (0,0) -- (2,0) -- (2,2) -- (0,2) -- (0,0);
  \draw[gray] (0,1) -- (2,1);
  \draw[gray] (1,0) -- (1,2);

  \node[circle,draw,inner sep=1pt] at (.5,1.5) (beta) {$\beta$};
  \node[circle,draw,inner sep=2pt] at (1.5,.5) (alpha) {$\alpha$};
  \draw[red] (beta) -- (2,1.5);
  \draw (beta) -- (.5,0);
  \draw[red] (alpha) -- (2,.5);
  \draw (alpha) -- (1.5,0);

  \end{scope}

  \node at (7,1) {$\simeq$};

  \begin{scope}[xshift=8cm]
    \draw[gray] (0,0) -- (2,0) -- (2,2) -- (0,2) -- (0,0);
  \draw[gray] (1,1) -- (2,1);
  \draw[gray] (1,0) -- (1,2);

  \node[circle,draw,inner sep=1pt] at (.5,1.5) (beta) {$\beta$};
  \node[circle,draw,inner sep=2pt] at (1.5,.5) (alpha) {$\alpha$};
  \draw[red] (beta) -- (2,1.5);
  \draw (beta) -- (.5,0);
  \draw[red] (alpha) -- (2,.5);
  \draw (alpha) -- (1.5,0);
  \end{scope}

  \node at (-1.3,1) {$\beta \star_{(0,1,0)} \alpha =$};
  \node at (11.3,1) {$= \beta \star_{(1,0,1)} \alpha$};
\end{tikzpicture}

  \caption{Two equivalent gluing positions in Lemma~\ref{lemma:gluing-no-input}}
  \label{fig:equiv-gluing-positions}
\end{figure}
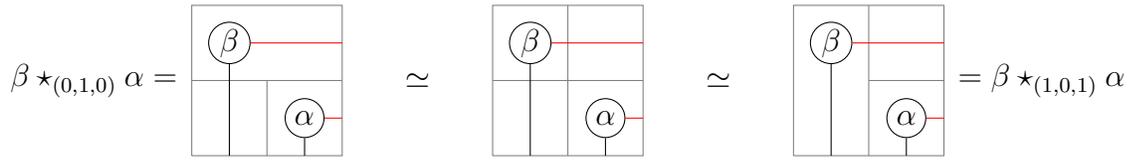

\begin{figure}
\begin{subfigure}{0.4\textwidth}
\begin{tikzpicture}[scale=0.8]
    \storeedgecolor{1}{red}
    \startdiagram{5}

    \foreach \x in {0,...,4} {
      \node at ($(-2.5,0-\x)$) (desc) {$h = \x$};
      \draw[dashed,gray] (desc) -- ($(desc)+(4.5,0)$);
    }

  \drawinitialstrands{5}

  \diagslice{1}{2}{2}

  \diagslice{2}{2}{1}

  \diagslice{0}{2}{1}

  \diagslice{1}{1}{0}

  \storeedgecolor{0}{red}
  \storeedgecolor{1}{red}
  \storeedgecolor{6}{red}
  \storeedgecolor{7}{red}
  \finishdiagram
  \begin{scope}[every node/.style={circle,fill=white,draw=black,inner sep=1.5pt}]
    \node at (v0) {$\alpha$};
    \node at (v1) {$\beta$};
    \node at (v2) {$\gamma$};
    \node at (v3) {$\delta$};
  \end{scope}
\end{tikzpicture}
\end{subfigure}
\begin{subfigure}{0.4\textwidth}
\begin{tikzpicture}
  \node at (-1,1.5) {$p_0(\phi) = $};

  \draw[gray] (-.25,0) -- (-.25,3.25) -- (3,3.25) -- (3,3) -- (0,3) -- (0,0) -- (-.25,0);
  \draw[gray] (-.25,3) -- (0,3);
  \draw[gray] (0,3) -- (0,3.25);
  \draw[red] (0,.5) -- (-.25,.5);
  \draw[red] (0,2.5) -- (-.25,2.5);
  \draw (.5,3) -- (.5,3.25);
  \draw (2.5,3) -- (2.5,3.25);
\end{tikzpicture}
\end{subfigure}
\begin{subfigure}{0.4\textwidth}
\begin{tikzpicture}
  \node at (-1,1.5) {$p_1(\phi) = $};

  \draw[gray] (-.25,0) -- (-.25,3.25) -- (3,3.25) -- (3,3) -- (0,3) -- (0,0) -- (-.25,0);
  \draw[gray] (-.25,3) -- (0,3);
  \draw[gray] (0,3) -- (0,3.25);
  \draw[red] (0,.5) -- (-.25,.5);
  \draw[red] (0,2.5) -- (-.25,2.5);
  \draw (.5,3) -- (.5,3.25);
  \draw (2.5,3) -- (2.5,3.25);

  \draw[gray] (0,2) -- (2,2) -- (2,3);
  \draw[gray] (2,2) -- (3,2) -- (3,3);
  \node[circle,draw,inner sep=2pt,scale=.8] at (.75,2.5) (alpha) {$\alpha$};
  \draw[red] (0,2.5) -- (alpha);
  \draw (.5,3) edge[bend right=10] (alpha);
  \draw (alpha) edge[bend right=10] (.5,2);
  \draw (alpha) edge[bend left=10] (1.5,2);

  \draw (2.5,3) -- (2.5,2);

\end{tikzpicture}
\end{subfigure}
\begin{subfigure}{0.4\textwidth}
\begin{tikzpicture}
  \node at (-1,1.5) {$p_2(\phi) = $};

  \draw[gray] (-.25,0) -- (-.25,3.25) -- (3,3.25) -- (3,3) -- (0,3) -- (0,0) -- (-.25,0);
  \draw[gray] (-.25,3) -- (0,3);
  \draw[gray] (0,3) -- (0,3.25);
  \draw[red] (0,.5) -- (-.25,.5);
  \draw[red] (0,2.5) -- (-.25,2.5);
  \draw (.5,3) -- (.5,3.25);
  \draw (2.5,3) -- (2.5,3.25);

  \draw[gray] (0,2) -- (2,2) -- (2,3);
  \draw[gray] (2,2) -- (3,2) -- (3,3);
  \node[circle,draw,inner sep=2pt,scale=.8] at (.75,2.5) (alpha) {$\alpha$};
  \draw[red] (0,2.5) -- (alpha);
  \draw (.5,3) edge[bend right=10] (alpha);
  \draw (alpha) edge[bend right=10] (.5,2);
  \draw (alpha) edge[bend left=10] (1.5,2);

  \draw (2.5,3) -- (2.5,2);

  \draw[gray] (1,2) -- (1,1) -- (0,1);
  \draw[gray] (3,2) -- (3,1) -- (1,1);
  \node[circle,draw,inner sep=1pt,scale=.8] at (2,1.4) (beta) {$\beta$};
  \draw (.5,2) -- (.5,1);
  \draw (1.5,2) edge[bend right=10] (beta);
  \draw (2.5,2) edge[bend left=10] (beta);
  \draw[red] (beta) -- (3,1.4);
\end{tikzpicture}
\end{subfigure}
\begin{subfigure}{0.4\textwidth}
\begin{tikzpicture}
  \node at (-1,1.5) {$p_3(\phi) = $};

  \draw[gray] (-.25,0) -- (-.25,3.25) -- (3,3.25) -- (3,3) -- (0,3) -- (0,0) -- (-.25,0);
  \draw[gray] (-.25,3) -- (0,3);
  \draw[gray] (0,3) -- (0,3.25);
  \draw[red] (0,.5) -- (-.25,.5);
  \draw[red] (0,2.5) -- (-.25,2.5);
  \draw (.5,3) -- (.5,3.25);
  \draw (2.5,3) -- (2.5,3.25);

  \draw[gray] (0,2) -- (2,2) -- (2,3);
  \draw[gray] (2,2) -- (3,2) -- (3,3);
  \node[circle,draw,inner sep=2pt,scale=.8] at (.75,2.5) (alpha) {$\alpha$};
  \draw[red] (0,2.5) -- (alpha);
  \draw (.5,3) edge[bend right=10] (alpha);
  \draw (alpha) edge[bend right=10] (.5,2);
  \draw (alpha) edge[bend left=10] (1.5,2);

  \draw (2.5,3) -- (2.5,2);

  \draw[gray] (1,2) -- (1,1) -- (0,1);
  \draw[gray] (3,2) -- (3,1) -- (1,1);
  \node[circle,draw,inner sep=1pt,scale=.8] at (2,1.4) (beta) {$\beta$};
  \draw (.5,2) -- (.5,1);
  \draw (1.5,2) edge[bend right=10] (beta);
  \draw (2.5,2) edge[bend left=10] (beta);
  \draw[red] (beta) -- (3,1.4);

  \draw[gray] (0,0) -- (3,0) -- (3,1);
  \node[circle,draw,inner sep=1pt,scale=.8] at (1,.5) (gamma) {$\gamma$};
  \draw (.5,1) edge[bend right=10] (gamma);
  \draw[red] (0,.5) -- (gamma);
  \draw[red] (gamma) -- (3,.5);
\end{tikzpicture}
\end{subfigure}
\begin{subfigure}{0.4\textwidth}
\begin{tikzpicture}
  \node at (-1,1.5) {$p_4(\phi) = $};

  \draw[gray] (-.25,0) -- (-.25,3.25) -- (3,3.25) -- (3,3) -- (0,3) -- (0,0) -- (-.25,0);
  \draw[gray] (-.25,3) -- (0,3);
  \draw[gray] (0,3) -- (0,3.25);
  \draw[red] (0,.5) -- (-.25,.5);
  \draw[red] (0,2.5) -- (-.25,2.5);
  \draw (.5,3) -- (.5,3.25);
  \draw (2.5,3) -- (2.5,3.25);

  \draw[gray] (0,2) -- (2,2) -- (2,3);
  \draw[gray] (2,2) -- (3,2) -- (3,3);
  \node[circle,draw,inner sep=2pt,scale=.8] at (.75,2.5) (alpha) {$\alpha$};
  \draw[red] (0,2.5) -- (alpha);
  \draw (.5,3) edge[bend right=10] (alpha);
  \draw (alpha) edge[bend right=10] (.5,2);
  \draw (alpha) edge[bend left=10] (1.5,2);

  \draw (2.5,3) -- (2.5,2);

  \draw[gray] (1,2) -- (1,1) -- (0,1);
  \draw[gray] (3,2) -- (3,1) -- (1,1);
  \node[circle,draw,inner sep=1pt,scale=.8] at (2,1.4) (beta) {$\beta$};
  \draw (.5,2) -- (.5,1);
  \draw (1.5,2) edge[bend right=10] (beta);
  \draw (2.5,2) edge[bend left=10] (beta);
  \draw[red] (beta) -- (3,1.4);

  \draw[gray] (0,0) -- (3,0) -- (3,1);
  \node[circle,draw,inner sep=1pt,scale=.8] at (1,.5) (gamma) {$\gamma$};
  \draw (.5,1) edge[bend right=10] (gamma);
  \draw[red] (0,.5) -- (gamma);
  \draw[red] (gamma) -- (3,.5);

  \draw[gray] (3,0) -- (4,0) -- (4,1) -- (3,1);
  \draw[gray] (4,1) -- (4,3.25) -- (3,3.25);
  \node[circle,draw,inner sep=1pt,scale=.8] at (3.5,1.4) (delta) {$\delta$};
  \draw[red] (3,.5) -- (4,.5);
  \draw[red] (3,1.4) -- (delta);
\end{tikzpicture}
\end{subfigure}

  \caption{Inductive construction of $p(\phi)$}
\end{figure}
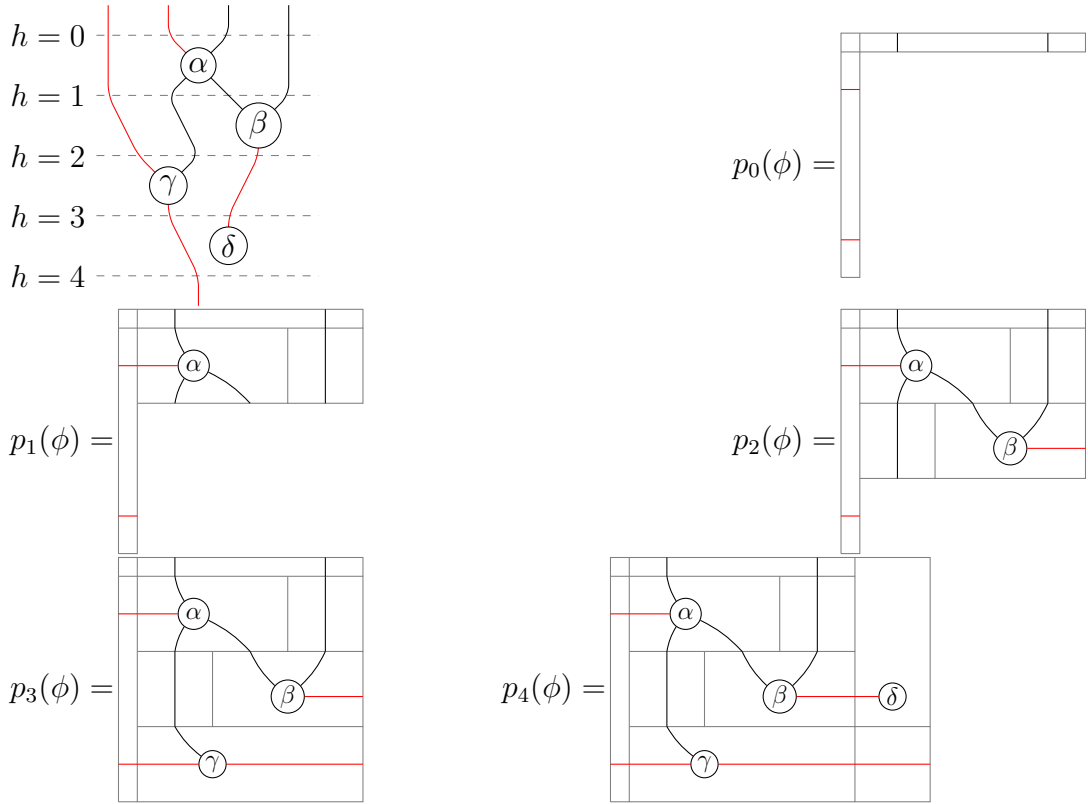

\begin{lemma} \label{lemma:p-functorial}
  For all 2-cell diagrams $\mu, \nu \in S_2$ such that \begin{tikzpicture}[baseline=2.5ex,scale=0.5,every node/.style={scale=0.7}]
    \draw[gray] (0,0) rectangle (2,2);
  \node[circle,draw,inner sep=1.5pt] at (.66,1.5) (mu) {$\mu$};
  \node[circle,draw,inner sep=2pt] at (1.33,.5) (nu) {$\nu$};
  \draw[red] (.25,2) -- (mu) -- (nu) -- (1.75,0);
  \draw (.25,0) -- (mu) -- (.8,2);
  \draw (1.2,0) -- (nu) -- (1.75,2);
  \end{tikzpicture} is defined,
  $p(\text
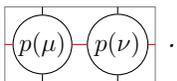
{\begin{tikzpicture}[baseline=2.5ex,scale=0.5,every node/.style={scale=0.7}]
    \draw[gray] (0,0) rectangle (2,2);
  \node[circle,draw,inner sep=1.5pt] at (.66,1.5) (mu) {$\mu$};
  \node[circle,draw,inner sep=2pt] at (1.33,.5) (nu) {$\nu$};
  \draw[red] (.25,2) -- (mu) -- (nu) -- (1.75,0);
  \draw (.25,0) -- (mu) -- (.8,2);
  \draw (1.2,0) -- (nu) -- (1.75,2);
  \end{tikzpicture}}) \simeq \text{
    \begin{tikzpicture}[baseline=2.55ex,scale=1,every node/.style={scale=.8}]
        \draw[gray] (0,0) rectangle (2,1);
  \node[circle,draw,inner sep=1pt] at (.5,.5) (mu) {$p(\mu)$};
  \node[circle,draw,inner sep=1.5pt] at (1.5,.5) (nu) {$p(\nu)$};
  \draw[red] (0,.5) -- (mu) -- (nu) -- (2,.5);
  \draw (.5,0) -- (mu) -- (.5,1);
  \draw (1.5,0) -- (nu) -- (1.5,1);
    \end{tikzpicture}
  }$.

  Similarly, if \begin{tikzpicture}[baseline=2.5ex,scale=0.5,every node/.style={scale=0.7}]
      \draw[gray] (0,0) rectangle (2,2);
  \node[circle,draw,inner sep=1.5pt] at (.66,1.5) (mu) {$\mu$};
  \node[circle,draw,inner sep=2pt] at (1.33,.5) (nu) {$\nu$};
  \draw (.25,2) -- (mu) -- (nu) -- (1.75,0);
  \draw[red] (.25,0) -- (mu) -- (.8,2);
  \draw[red] (1.2,0) -- (nu) -- (1.75,2);
  \end{tikzpicture} is defined, then
  $p(\text{\begin{tikzpicture}[baseline=2.5ex,scale=0.5,every node/.style={scale=0.7}]
      \draw[gray] (0,0) rectangle (2,2);
  \node[circle,draw,inner sep=1.5pt] at (.66,1.5) (mu) {$\mu$};
  \node[circle,draw,inner sep=2pt] at (1.33,.5) (nu) {$\nu$};
  \draw (.25,2) -- (mu) -- (nu) -- (1.75,0);
  \draw[red] (.25,0) -- (mu) -- (.8,2);
  \draw[red] (1.2,0) -- (nu) -- (1.75,2);
  \end{tikzpicture}}) \simeq \text{\begin{tikzpicture}[baseline=2.5ex,scale=1,every node/.style={scale=0.8}]
   \begin{scope}[xshift=1cm,yshift=-.5cm,rotate=90]
  \draw[gray] (0,0) rectangle (2,1);
  \node[circle,draw,inner sep=2pt] at (.5,.5) (mu) {$p(\nu)$};
  \node[circle,draw,inner sep=1.5pt] at (1.5,.5) (nu) {$p(\mu)$};
  \draw (0,.5) -- (mu) -- (nu) -- (2,.5);
  \draw[red] (.5,0) -- (mu) -- (.5,1);
  \draw[red] (1.5,0) -- (nu) -- (1.5,1);
   \end{scope}
   \end{tikzpicture}
}$.
\end{lemma}

\begin{proof}
  By duality let us prove the result for the first case, horizontal composition.
  Let $\phi = \text{\begin{tikzpicture}[baseline=2.5ex,scale=0.5,every node/.style={scale=0.7}]
    \draw[gray] (0,0) rectangle (2,2);
  \node[circle,draw,inner sep=1.5pt] at (.66,1.5) (mu) {$\mu$};
  \node[circle,draw,inner sep=2pt] at (1.33,.5) (nu) {$\nu$};
  \draw[red] (.25,2) -- (mu) -- (nu) -- (1.75,0);
  \draw (.25,0) -- (mu) -- (.8,2);
  \draw (1.2,0) -- (nu) -- (1.75,2);
  \end{tikzpicture}}$.
  Let $h$ be the level between $\mu$ and $\nu$ in $\phi$.

  Assume first that the red edge connecting $\mu$ and $\nu$ is not
  empty (it is not an identity vertical morphism). Then $p_h(\phi) =
  \text{\begin{tikzpicture}[baseline=2ex,scale=0.8,every
        node/.style={scale=0.7}] \draw[gray] (-.25,0) -- (0,0) --
      (0,1) -- (2,1) -- (2,1.25) -- (-.25,1.25) -- (-.25,0);
      \draw[gray] (0,0) -- (1,0) -- (1,1); \draw[gray] (-.25,1) --
      (0,1) -- (0,1.25); \node[circle,draw,inner sep=1pt] at (.5,.5)
      (mu) {$p(\mu)$}; \draw[red] (-.25,.5) -- (mu) -- (1,.5); \draw
      (.5,1.25) -- (mu) -- (.5,0); \draw (1.5,1.25) -- (1.5,1);
  \end{tikzpicture}}$ and $p(\phi)$ is obtained from $p_h(\phi)$
  by gluing on it the generators in $\nu$.
  Since the vertical codomain of $\mu$ passes to the left of $\nu$, these generators are glued on positions $(k,i,j)$ with $k > 0$.
  Performing these gluings on an empty diagram gives $p(\nu)$,
  so $p(\phi)$ is equivalent to the required double diagram.

  If there is no red edge connecting $\mu$ to $\nu$ then
  $p_h(\phi) = \text{\begin{tikzpicture}[baseline=2ex,scale=0.8,every
        node/.style={scale=0.7}] \draw[gray] (-.25,0) -- (0,0) --
      (0,1) -- (2,1) -- (2,1.25) -- (-.25,1.25) -- (-.25,0);
      \draw[gray] (0,0) -- (1,0) -- (1,1);
      \draw[gray] (1,0) -- (2,0) -- (2,1);
      \draw[gray] (-.25,1) -- (0,1) -- (0,1.25); \node[circle,draw,inner sep=1pt] at (.5,.5)
      (mu) {$p(\mu)$}; \draw[red] (-.25,.5) -- (mu);
      \draw (.5,1.25) -- (mu) -- (.5,0);
      \draw (1.5,1.25) -- (1.5,0);
  \end{tikzpicture}}$ and $p(\phi)$ is obtained from
  $p_h(\phi)$ by gluing the generators in $\nu$ on the second part of its vertical codomain, so it can again be rewritten into the required form by unitality.
\end{proof}

\begin{lemma} \label{lemma:translation-inverse}
  For any 2-cell diagram $\phi \in S_d$, $p(t(\phi)) \simeq \phi$.
\end{lemma}

\begin{proof}
  By induction on $\phi$. If $\phi$ is a generator or the identity,
  the result holds.

  If $\phi$ is a horizontal or vertical composition, then we use
  Lemma~\ref{lemma:p-functorial} and the induction hypothesis of
  the composed diagrams:
  
  $$p(t(\text{
\begin{tikzpicture}[baseline=2.55ex,scale=1,every node/.style={scale=.8}]
        \draw[gray] (0,0) rectangle (2,1);
  \node[circle,draw,inner sep=1.5pt] at (.5,.5) (mu) {$\mu$};
  \node[circle,draw,inner sep=2pt] at (1.5,.5) (nu) {$\nu$};
  \draw[red] (0,.5) -- (mu) -- (nu) -- (2,.5);
  \draw (.5,0) -- (mu) -- (.5,1);
  \draw (1.5,0) -- (nu) -- (1.5,1);
    \end{tikzpicture}
  })) = p(\text{\begin{tikzpicture}[baseline=3.2ex,scale=0.6,every node/.style={scale=0.5}]
    \draw[gray] (0,0) rectangle (2,2);
  \node[circle,draw,inner sep=1pt] at (.66,1.5) (mu) {$t(\mu)$};
  \node[circle,draw,inner sep=1.5pt] at (1.33,.5) (nu) {$t(\nu)$};
  \draw[red] (.25,2) -- (mu) -- (nu) -- (1.75,0);
  \draw (.25,0) -- (mu) -- (.8,2);
  \draw (1.2,0) -- (nu) -- (1.75,2);
  \end{tikzpicture}}) \simeq \text{
\begin{tikzpicture}[baseline=2.55ex,scale=1,every node/.style={scale=.7}]
        \draw[gray] (0,0) rectangle (2,1);
  \node[circle,draw,inner sep=.5pt] at (.5,.5) (mu) {$p(t(\mu))$};
  \node[circle,draw,inner sep=.5pt] at (1.5,.5) (nu) {$p(t(\nu))$};
  \draw[red] (0,.5) -- (mu) -- (nu) -- (2,.5);
  \draw (.5,0) -- (mu) -- (.5,1);
  \draw (1.5,0) -- (nu) -- (1.5,1);
    \end{tikzpicture}
  } \simeq \text{
\begin{tikzpicture}[baseline=2.55ex,scale=1,every node/.style={scale=.8}]
        \draw[gray] (0,0) rectangle (2,1);
  \node[circle,draw,inner sep=1.5pt] at (.5,.5) (mu) {$\mu$};
  \node[circle,draw,inner sep=2pt] at (1.5,.5) (nu) {$\nu$};
  \draw[red] (0,.5) -- (mu) -- (nu) -- (2,.5);
  \draw (.5,0) -- (mu) -- (.5,1);
  \draw (1.5,0) -- (nu) -- (1.5,1);
    \end{tikzpicture}
  }$$
\end{proof}

\begin{lemma} \label{lemma:faithful-translation}
  Let $\phi, \phi'$ be admissible diagrams in $S_2$ with $\phi \sim
  \phi'$. Then $p(\phi) \simeq p(\phi')$.
\end{lemma}

\begin{proof}
  By induction we can assume that $\phi$ and $\phi'$ are related by a
  single exchange, swapping the generators between levels $h-1$, $h$
  and $h+1$. Let $\alpha$ be the generator between levels $h-1$ and
  $h$ and $\beta$ the one between $h$ and $h+1$.  It suffices to check
  that $p_{h-1}(\phi) \star_l \alpha \star_{l'} \beta \simeq
  p_{h-1}(\phi) \star_m \beta \star_{m'} \alpha$ where $l, l'$ are the
  gluing positions for the generators in $\phi$ and $m, m'$ are their
  counterparts in $\phi'$.  By a tedious case analysis one can check
  that because the generators at these slices can be exchanged, this
  ensures that the induced gluing positions are disjoint, such that the equivalence above either holds trivially (the partial tilings being syntactically equal) or via equivalences analogous to those of Figure~\ref{fig:equiv-gluing-positions}.
\end{proof}

\section{Word problem} \label{sec:word-problem}

We can use the translation defined in the previous section to
solve the word problem for double categories:

\begin{thm} \label{thm:word-problem}
  Let $S$ be a double signature.
  The word problem for 2-cells in the free double category $S_d$
  can be solved in $O(v e)$, where $v$ is the number of generators
  in the expressions and $e$ the number of connecting edges between them.
\end{thm}

\begin{proof}
  Given two diagrams $\phi, \phi'$ in $S_d$, we can compute their
  translation $t(\phi), t(\phi')$ in linear time. Then,
  we can check if these diagrams are equivalent as 2-cells in $S_2$
  using the algorithm of \cite{delpeuch2018normalization}, in $O(v e)$.
  As $p$ is faithful (Lemma~\ref{lemma:faithful-translation}), this
  determines if $\phi$ and $\phi'$ are equivalent in $S_d$.
\end{proof}

\section{Conclusion} \label{sec:pinwheel}

We have solved the word problem for free double categories
by reducing it to the word problem for free 2-categories.

Interestingly, our translation $p$ from the free 2-category to the
free double category works for all admissible diagrams, and
admissibility is a simple condition on the domain and codomain.  We
are not requiring any global condition such as binary composability on
the 2-cell. As a consequence, this translation $p$ can produce tilings
which are not binary composable.

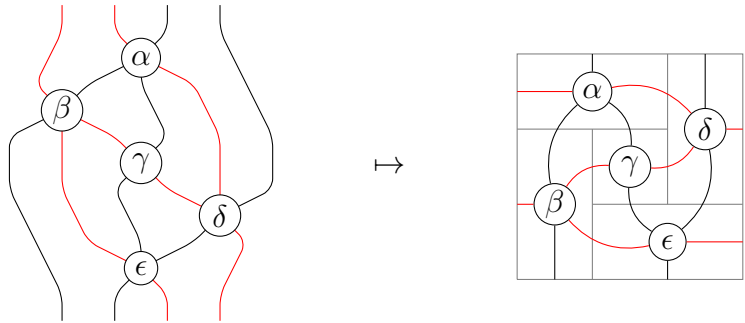
\[
\begin{tikzpicture}
  \begin{scope}[xshift=5cm,yshift=-3.3cm]

    \node at (-1.7,1.5) {$\mapsto$};
  \draw[gray] (0,0) -- (0,3) -- (3,3) -- (3,0) -- (0,0);
  \draw[gray] (0,2) -- (2,2);
  \draw[gray] (2,3) -- (2,1);
  \draw[gray] (1,1) -- (3,1);
  \draw[gray] (1,0) -- (1,2);

  \begin{scope}[every node/.style={circle,draw,inner sep=2pt,fill=white}]
    \node at (1.5,1.5) (gamma) {$\gamma$};
    \node at (1,2.5) (alpha) {$\alpha$};
    \node[inner sep=1pt] at (.5,1) (beta) {$\beta$};
    \node at (2.5,2) (delta) {$\delta$};
    \node[inner sep=2.5pt] at (2,.5) (epsilon) {$\epsilon$};
    \draw (alpha) edge[bend left] (gamma) edge [bend right] (beta);
    \draw (epsilon) edge[bend left] (gamma) edge[bend right] (delta);
    \draw[red] (beta) edge[bend left] (gamma) edge[bend right] (epsilon);
    \draw[red] (delta) edge[bend left] (gamma) edge[bend right] (alpha);

    \draw[red] (0,1) -- (beta);
    \draw (1,3) -- (alpha);
    \draw[red] (3,2) -- (delta);
    \draw (2,0) -- (epsilon);

    \draw[red] (0,2.5) -- (alpha);
    \draw (.5,0) -- (beta);
    \draw[red] (3,.5) -- (epsilon);
    \draw (2.5,3) -- (delta);
    
  \end{scope}
  \end{scope}

  \begin{scope}[scale=0.7]
  \startdiagram{5}

  \drawinitialstrands{5}

  \diagslice{1}{2}{3}

  \diagslice{0}{2}{3}

  \diagslice{2}{2}{2}

  \diagslice{3}{3}{2}

  \diagslice{1}{3}{2}

  \storeedgecolor{0}{red}
  \storeedgecolor{1}{red}
  \storeedgecolor{6}{red}
  \storeedgecolor{8}{red}
  \storeedgecolor{9}{red}
  \storeedgecolor{11}{red}
  \storeedgecolor{13}{red}
  \storeedgecolor{15}{red}
  \finishdiagram

  \begin{scope}[every node/.style={circle,draw,inner sep=2pt,fill=white}]
    \node at (v0) {$\alpha$};
    \node[inner sep=1pt] at (v1) {$\beta$};
    \node at (v2) {$\gamma$};
    \node at (v3) {$\delta$};
    \node at (v4) {$\epsilon$};
  \end{scope}
  \end{scope}
\end{tikzpicture}
\]

It is therefore tempting to extend the forward translation $t$ to
double category diagrams which are not binary composable. By
the characterization of \cite{dawson1995forbiddensuborder} of non-composable diagrams, it is sufficient to translate the two
pinwheels: by induction, all diagrams could then be interpreted.
However, there could potentially be multiple ways to decompose a
diagram as a tree of binary and pinwheel composites, so to define $t$
properly we would need an equivalent of the general associativity
result of \cite{dawson1993general} with pinwheel composition. That
would only be possible given an appropriate notion of equivalence, which would amount to developing a notion of ``double category with
pinwheels''. This does not strike us as a particularly useful notion
as it would be rather complicated, with four different composition
operators and many axioms to relate their applications, only to
represent planar systems.

What this really means is that free
2-categories already provide the appropriate notion of ``free double
category with pinwheel composites'', in the sense that they capture the desired combinatorics with a much simpler axiomatization.

This fact has been observed at an intuitive level by \cite{reutter2019biunitary}
who modeled biunitary connections in a 2-category rather than a double category,
by using the same rotation. They noticed that biunitaries forming a pinwheel
pattern could be composed into a new biunitary. As double categories are not equipped with
such a composition, a 2-categorical model provides a more useful representation.
Modelling biunitaries in double categories would artificially forbid pinwheel composites
which are actually allowed physically.
We suspect that other uses of double categories, for instance in computer science~\citep{bruni2002symmetric},
could be recast in 2-categories without loss of expressivity, as they do not rely on the exclusion
of pinwheel composites.

\bibliographystyle{plainnat}

\begin{thebibliography}{17}
\providecommand{\natexlab}[1]{#1}
\providecommand{\url}[1]{\texttt{#1}}
\expandafter\ifx\csname urlstyle\endcsname\relax
  \providecommand{\doi}[1]{doi: #1}\else
  \providecommand{\doi}{doi: \begingroup \urlstyle{rm}\Url}\fi

\bibitem[Bruni et~al.(2002)Bruni, Meseguer, and Montanari]{bruni2002symmetric}
Roberto Bruni, Jos{\'e} Meseguer, and Ugo Montanari.
\newblock Symmetric monoidal and cartesian double categories as a semantic
  framework for tile logic.
\newblock \emph{Mathematical Structures in Computer Science}, 12\penalty0
  (1):\penalty0 53--90, February 2002.
\newblock ISSN 1469-8072, 0960-1295.
\newblock \doi{10.1017/S0960129501003462}.

\bibitem[Charney(1992)]{charney1992artin}
Ruth Charney.
\newblock Artin groups of finite type are biautomatic.
\newblock \emph{Mathematische Annalen}, 292\penalty0 (1):\penalty0 671--683,
  March 1992.
\newblock ISSN 1432-1807.
\newblock \doi{10.1007/BF01444642}.

\bibitem[Dawson and Par{\'e}(1993)]{dawson1993characterizing}
R.~Dawson and R.~Par{\'e}.
\newblock Characterizing tileorders.
\newblock \emph{Order}, 10\penalty0 (2):\penalty0 111--128, June 1993.
\newblock ISSN 1572-9273.
\newblock \doi{10.1007/BF01111295}.

\bibitem[Dawson et~al.(2004)Dawson, Par{\'e}, and Pronk]{dawson2004free}
R.~J.~M. Dawson, R.~Par{\'e}, and D.~A. Pronk.
\newblock Free extensions of double categories.
\newblock \emph{Cahiers de topologie et g{\'e}om{\'e}trie diff{\'e}rentielle
  cat{\'e}goriques}, 45\penalty0 (1):\penalty0 35--80, 2004.

\bibitem[Dawson(1995)]{dawson1995forbiddensuborder}
Robert Dawson.
\newblock A forbidden-suborder characterization of binarily-composable diagrams
  in double categories.
\newblock \emph{Theory and Applications of Categories [electronic only]},
  1:\penalty0 146--155, 1995.
\newblock ISSN 1201-561X.

\bibitem[Dawson and Pare(1993)]{dawson1993general}
Robert Dawson and Robert Pare.
\newblock General associativity and general composition for double categories.
\newblock \emph{Cahiers de Topologie et G{\'e}om{\'e}trie Diff{\'e}rentielle
  Cat{\'e}goriques}, 34\penalty0 (1):\penalty0 57--79, 1993.

\bibitem[Dawson and Par{\'e}(2002)]{dawson2002what}
Robert Dawson and Robert Par{\'e}.
\newblock What is a free double category like?
\newblock \emph{Journal of Pure and Applied Algebra}, 168\penalty0
  (1):\penalty0 19--34, March 2002.
\newblock ISSN 0022-4049.
\newblock \doi{10.1016/S0022-4049(01)00049-4}.

\bibitem[Delpeuch and Vicary(2018)]{delpeuch2018normalization}
Antonin Delpeuch and Jamie Vicary.
\newblock Normalization for planar string diagrams and a quadratic equivalence
  algorithm.
\newblock \emph{arXiv:1804.07832 [cs]}, April 2018.

\bibitem[Ehresmann(1963)]{ehresmann1963categories}
Charles Ehresmann.
\newblock Cat{\'e}gories structur{\'e}es.
\newblock \emph{Annales scientifiques de l'{\'E}cole Normale Sup{\'e}rieure},
  80\penalty0 (4):\penalty0 349--426, 1963.

\bibitem[Epstein et~al.(1992)Epstein, Paterson, Cannon, Holt, Levy, and
  Thurston]{epstein1992word}
David Epstein, Mike~S Paterson, James~W Cannon, Derek~F Holt, Silvio~V Levy,
  and William~P Thurston.
\newblock \emph{Word Processing in Groups}.
\newblock {AK Peters, Ltd.}, 1992.

\bibitem[Fiore et~al.(2008)Fiore, Paoli, and Pronk]{fiore2008model}
Thomas~M Fiore, Simona Paoli, and Dorette Pronk.
\newblock Model structures on the category of small double categories.
\newblock \emph{Algebraic \& Geometric Topology}, 8\penalty0 (4):\penalty0
  1855--1959, October 2008.
\newblock ISSN 1472-2739, 1472-2747.
\newblock \doi{10.2140/agt.2008.8.1855}.

\bibitem[Heidemann et~al.(2019)Heidemann, Hu, and Vicary]{homotopyio-tool}
Lukas Heidemann, Nick Hu, and Jamie Vicary.
\newblock Homotopy.io.
\newblock 2019.
\newblock \doi{10.5281/zenodo.2540764}.

\bibitem[Mosher(1995)]{mosher1995mapping}
Lee Mosher.
\newblock Mapping {{Class Groups}} are {{Automatic}}.
\newblock \emph{Annals of Mathematics}, 142\penalty0 (2):\penalty0 303--384,
  1995.
\newblock ISSN 0003-486X.
\newblock \doi{10.2307/2118637}.

\bibitem[Myers(2016)]{myers2016string}
David~Jaz Myers.
\newblock String {{Diagrams For Double Categories}} and {{Equipments}}.
\newblock \emph{arXiv:1612.02762 [math]}, December 2016.

\bibitem[Reutter and Vicary(2019)]{reutter2019biunitary}
David~J. Reutter and Jamie Vicary.
\newblock Biunitary constructions in quantum information.
\newblock \emph{Higher Structures}, 3\penalty0 (1):\penalty0 109--154, 2019.
\newblock ISSN 2209-0606

\bibitem[Selinger(2011)]{selinger2011survey}
Peter Selinger.
\newblock A survey of graphical languages for monoidal categories.
\newblock In \emph{New {{Structures}} for {{Physics}}}, volume 813 of
  \emph{Lecture {{Notes}} in {{Physics}}}, pages 289--233. {Springer}, 2011.

\bibitem[Squier(1987)]{squier1987word}
Craig~C. Squier.
\newblock Word problems and a homological finiteness condition for monoids.
\newblock \emph{Journal of Pure and Applied Algebra}, 49\penalty0
  (1-2):\penalty0 201--217, November 1987.
\newblock ISSN 0022-4049.
\newblock \doi{10.1016/0022-4049(87)90129-0}.

\end{thebibliography}

\end{document}